\documentclass[letterpaper,11pt,reqno]{amsart}
\usepackage{hyperref}

\hypersetup{
     colorlinks   = true,
     citecolor    = black,
     linkcolor    = blue
} 

\usepackage{amsmath}
\usepackage{amssymb}
\usepackage{marginnote}
\usepackage{enumitem}
\usepackage{color}
\usepackage{url}
\usepackage[usenames,dvipsnames]{xcolor}
\usepackage{graphicx}
\usepackage{verbatim}
\usepackage{tikz}
\makeatletter % makes "@" a normal character so that it can be used in macros
\def\@cite#1#2{{\m@th\upshape\bfseries%
[{#1\if@tempswa{\m@th\upshape\mdseries, #2}\fi}]}}
\makeatother % returns "@" to its normal state

\usepackage[normalem]{ulem}
\usepackage{cancel}

\theoremstyle{plain}
\newtheorem{thm}{Theorem}[section]% subsection

\newtheorem{cor}[thm]{Corollary}

\newtheorem{lem}[thm]{Lemma}
\newtheorem{lemma}[thm]{Lemma}

\theoremstyle{definition}
\newtheorem{definition}[thm]{Definition}

\newtheorem{quest}[thm]{Question}

\newtheorem{rem}[thm]{Remark}

\numberwithin{equation}{subsection}
\newcommand{\nc}{\newcommand}
\newcommand{\rnc}{\renewcommand}

\nc\bA{\mathbb{A}}
\nc\bB{\mathbb{B}}
\nc\bC{\mathbb{C}}
\nc\bD{\mathbb{D}}
\nc\bE{\mathbb{E}}
\nc\bF{\mathbb{F}}
\nc\bG{\mathbb{G}}
\nc\bH{\mathbb{H}}
\nc\bI{\mathbb{I}}
\nc{\bJ}{\mathbb{J}} 
\nc\bK{\mathbb{K}}
\nc\bL{\mathbb{L}}
\nc\bM{\mathbb{M}}
\nc\bN{\mathbb{N}}
\nc\bO{\mathbb{O}}
\nc\bP{\mathbb{P}}
\nc\bQ{\mathbb{Q}}
\nc\bR{\mathbb{R}}
\nc\bS{\mathbb{S}}
\nc\bT{\mathbb{T}}
\nc\bU{\mathbb{U}}
\nc\bV{\mathbb{V}}
\nc\bW{\mathbb{W}}
\nc\bY{\mathbb{Y}}
\nc\bX{\mathbb{X}}
\nc\bZ{\mathbb{Z}}
\nc\wt{\widetilde}

\nc\cA{\mathcal{A}}
\nc\cB{\mathcal{B}}
\nc\cC{\mathcal{C}}
\nc\cD{\mathcal{D}}
\nc\cE{\mathcal{E}}
\nc\cF{\mathcal{F}}
\nc\cG{\mathcal{G}}
\nc\cH{\mathcal{H}}
\nc\cI{\mathcal{I}}
\nc{\cJ}{\mathcal{J}} 
\nc\cK{\mathcal{K}}
\nc\cL{\mathcal{L}}
\nc\cM{\mathcal{M}}
\nc\cN{\mathcal{N}}
\nc\cO{\mathcal{O}}
\nc\cP{\mathcal{P}}
\nc\cQ{\mathcal{Q}}
\nc\cR{\mathcal{R}}
\nc\cS{\mathcal{S}}
\nc\cT{\mathcal{T}}
\nc\cU{\mathcal{U}}
\nc\cV{\mathcal{V}}
\nc\cW{\mathcal{W}}
\nc\cY{\mathcal{Y}}
\nc\cX{\mathcal{X}}
\nc\cZ{\mathcal{Z}}

\newcommand{\defeq}{:=}

\nc{\Tm}{Teichm\"uller}
\nc{\Mod}{\text{Mod}}

\newcommand{\C}{\mathbb C}
\newcommand{\R}{\mathbb R}

\newcommand{\SL}{\mathrm{SL}}

\newcommand{\SO}{\mathrm{SO}}
\newcommand{\gradient}{\nabla}
\newcommand{\laplacian}{\Delta}

\rnc{\Re}{\text{Re}}
\rnc{\Im}{\text{Im}}
\nc{\var}{\text{var}}
\nc{\cov}{\text{covar}}
\nc{\deq}{\stackrel{d}{=}}

\newcommand{\bfH}{\textbf{H}}

\newcommand{\bfv}{\textbf{v}}

\newcommand{\E}[1]{\bE_{\hat{\nu}_\bP}\left[ #1 \right]}

\title{A Central Limit Theorem for the Kontsevich-Zorich Cocycle}

\author[H. Al-Saqban]{Hamid~Al-Saqban}
\address{Institut f\"ur Mathematik, Universit\"at Paderborn}
\email{hqs@umd.edu}

\author[G. Forni]{Giovanni~Forni}
\address{Department of Mathematics, University of Maryland, College Park}
\email{gforni@umd.edu}

\begin{document}

\begin{abstract}
We show that a central limit theorem holds for exterior powers of the Kontsevich-Zorich (KZ) cocycle. In particular, we show that, under the hypothesis that the top Lyapunov exponent on the exterior power is simple, a central limit theorem holds for the lift of the (leaf-wise) hyperbolic Brownian motion to any strongly irreducible, symplectic, $\SL(2,\bR)$-invariant subbundle, that is moreover symplectic-orthogonal to the so-called tautological subbundle. We then show that this implies that a central limit theorem holds for the lift of the \Tm~geodesic flow to the same bundle. 

For the random cocycle over the hyperbolic Brownian motion, we prove under the same hypotheses that the variance of the top exponent is strictly positive. For the deterministic cocycle over the 
\Tm~geodesic flow we prove that the variance is strictly positive only for the top exponent of the first exterior power (the KZ cocycle itself) under the hypothesis that its Lyapunov spectrum is simple.
\end{abstract}
\maketitle

\section{Introduction}

This paper concerns the Kontsevich-Zorich (KZ) cocycle, a much studied dynamical system in the field of Teichm\"uller dynamics. The KZ cocycle has played a major role in addressing multiple questions of physical interest, including, for example, the computation of diffusion rates on wind-tree models \cite{DHL14}, and it itself acts as a renormalizing dynamical system for straight-line flows on translation surfaces. We refer the  reader to the surveys \cite{Z06, FM13, W15}  for an introduction to this rich area of research. 

It is now well-established that Hodge theory, together with classical potential theory, can be brought to bear on this cocycle and its associated Lyapunov exponents, thanks to the pioneering works of M.~Kontsevich and A.~Zorich \cite{KZ97, Kon97}, later developed in~\cite{For02}. In these works, the Hodge norm was introduced
in Teichm\"uller dynamics, and in \cite{For02} it was proved that the logarithm of the Hodge norm is a subharmonic function on all exterior powers of the cocycle, hence the KZ cocycle has positive exponents on strata (see also \cite{For06}, \cite{For11}, \cite{FMZ11}, \cite{FMZ12}). 

The Hodge norm has since played a crucial role in the developments in Teichm\"uller dynamics, in particular in the study of the hyperbolicity properties of the KZ cocycle and of the Teichm\"uller flow, and related questions in the ergodic theory of translations flows and interval exchange transformations. A very partial list of landmark applications of the Hodge norm includes  \cite{ABEM12, EKZ14, EMM15,EM18, Fil16a, Fil16b, Fil17}.

The purpose of this paper is to show that probabilistic potential theory (and thus stochastic calculus), can be applied to study the oscillations of the Hodge norm of the KZ cocycle. In fact, we prove a  (non-commutative) Central Limit Theorem (CLT) for exterior powers of both the random and the deterministic KZ cocycles, and we moreover prove the non-degeneracy of the CLT for the random cocycles and for the first exterior power of the deterministic KZ cocycle (the KZ cocycle itself) under the natural dynamical assumptions of simplicity of the Lyapunov spectrum. Motivated by computer experiments and the results in \cite{Z96, Z97}, the simplicity of the KZ spectrum for all strata of the moduli space of Abelian differentials was conjectured by M.~Kontsevich and A.~Zorich in \cite{KZ97, Kon97}. It was then established by A.~Avila and M.~Viana in \cite{AV07}, and, in genus~$2$, for all $\SL(2,\R)$-invariant orbifolds, by M.~Bainbridge in \cite{B07}. 

The problem of finding oscillations of the KZ cocycle has been the subject of recent interest: in \cite{A21}, a mechanism to produce oscillations of the KZ cocycle was presented, where the basepoint is a fixed surface, and a more refined mechanism was developed by J.~Chaika, O.~Khalil and J.~Smillie in their work on the ergodic measures of the \Tm~ horocycle flow \cite{CKS21}. We expect that the deterministic central limit theorem presented here can be brought to bear on the scope of these results.

 The probabilistic ideas that inspired our approach, and their application to geodesic flows, go back to the work of Y.~Le Jan \cite{L94} (see also \cite{FL12}). We refer the reader to \cite{RY} for a comprehensive  introduction to stochastic calculus, including fundamental results exploited in this paper. 

The approach we follow to prove the CLT for the random cocycle also relies crucially on the analysis of the Brownian semigroup to solve a (leaf-wise) Poisson equation, by a method reminiscent of \cite{L95}. In fact we prove exponential mixing for the lift of the Brownian motion to the Hodge bundle, leveraging the exponential mixing of the \Tm~geodesic flow, due to A.~Avila, S.~Gou\"ezel and J.-C.~Yoccoz for strata \cite{AGY06} and to A.~Avila and S.~Gou\"ezel for all $\SL(2,\R)$-invariant orbifolds \cite{AG13}.
 
The CLT for the deterministic cocycle is then derived from the corresponding result for the random cocycle by a stopping time argument based in part on an asymptotic estimate due to A.~Ancona \cite{Anc90}.

We point out that in the setting of products of independent and identically distributed random matrices, the central limit theorem was  established, in varying levels of generality, by Bellman in \cite{B54}, H. Furstenberg and H. Kesten in \cite{FK60}, Tutubalin in \cite{T77}, Le Page in \cite{LP82}, Y. Guivarc'h and A. Raugi in \cite{GR85},  I. Ya. Golsheid and G. A. Margulis in \cite{GM89}, Hennion in \cite{H97}, Jan in \cite{J00}, and more recently, and under an optimal finite second moment condition, by Y.~Benoist and 
J.-F.~Quint in \cite{BQ16}. The central limit theorem was also established for solutions of linear stochastic differential equations with Markovian coefficients by P.~Bougerol in \cite{B88}. On the other hand, to the best of our knowledge, there are no comparable works in the setting of deterministic cocycles over (non-uniformly) hyperbolic flows, such as the one we treat here. We point to \cite{DKP21}, \cite{FK21}, \cite{PP22} for some results on the central limit theorem in this direction. However, the cited results are established for cocycles over shifts of finite type (SFT's), while it is well known that the Teichm\"uller flow has a symbolic representation as a suspension flow over a Markov shift on a countable alphabet. 
While not the original aim of this paper, we note that our work addresses, if ever so incrementally, this dearth of central limit theorem results for deterministic cocycles over non-uniformly hyperbolic flows (a related result for the KZ cocycle, based on the study of a transfer operator via anisotropic Banach spaces, has been recently announced by O.~Khalil). We finally remark that the positivity of the variance for the particular case of the (deterministic) KZ cocycle, that we prove in this paper under certain hypotheses, would remain a non-trivial question, even if a general theorem were available.

In another direction, the paper of J.~Daniels and B.~Deroin \cite{DD19} adapted the \Tm~dynamics methodology to more general compact K\"ahler manifolds, and one in which the methods in this paper are applicable, provided that we can prove existence of a solution to Poisson's equation for the corresponding Laplacian. 

In \cite{DFV17}, D.~Dolgopyat, B.~Fayad and I.~Vinogradov proved a central limit theorem for the Siegel transform of sufficiently regular observables for the diagonal action on the space of
lattices. Their methods are in fact much more general and  imply in particular a Central Limit Theorem for pushforwards of (unstable) unipotent arcs with respect to the uniform distribution on almost every unipotent orbit \cite[Theorem 7.1, Corollary 7.2]{DFV17}. It would be interesting to prove exponential mixing for the action of the Teichm\"uller flow on the projectivized Hodge bundle {(see also Question \ref{expmix}), with the aim of applying a multiplicative generalization of their results to the KZ cocycle.

After the appearance of a first
draft of our paper, F. Arana-Herrera and the second-named author in \cite{AF24} proved a central limit theorem for  sections of the Hodge bundle and a mixing central limit theorem for the Kontsevich-Zorich cocycle, with the results of this paper as a crucial input in a more general approach to
mixing limit distributions results.

\section{Statement of results}
Let $\pi: \bP(\textbf{H}) \to X$ be the projectivization of a strongly irreducible $\SL(2,\bR)$-invariant symplectic subbundle of the absolute (real) Hodge bundle over an $\SL(2,\bR)$ orbit closure $X$, whose fiber over each point in $X$ is $H^1(S,\bR)$, and with $\nu$ an ergodic $\SL(2,\bR)$-invariant probability measure on~$X$. The Kontsevich-Zorich cocycle is the lift of the $\SL(2,\bR)$ action  to~$\bP(\textbf{H})$, obtained by parallel transport with respect to the Gauss-Manin connection. Furthermore, the cocycle acts symplectically since it preserves the intersection form on $H^1(S,\bR)$. 

For our purposes, we will be concerned with $k$-th exterior powers $\textbf{H}^{(k)}$ of strongly irreducible invariant symplectic components $\textbf{H}$ of the Hodge bundle, which are symplectic orthogonal to the tautological subbundle (spanned for every $\omega \in X $ by $[\Re~\omega]$ and  $[\Im~\omega])$. 

We will denote by $\bP (\textbf{H}^{(k)})$ the projectivization
of the bundle $\textbf{H}^{(k)}$. It follows by a result of Bonatti-Eskin-Wilkinson \cite{BEW20} that any $P$-invariant ($P<\SL(2,\R)$ is the maximal parabolic subgroup generated by the diagonal subgroup and by the forward unstable unipotent subgroup) probability measure $\hat{\nu}$ on this bundle which projects to $\nu$ on $X$ is such that, for $\nu$-a.a. $\omega\in X$, the conditional measure on $\bP (\textbf{H}^{(k)}_\omega)$ is the Dirac mass supported on the unstable Oseledets subspace $E^+_k(\omega)\subset \textbf{H}^{(k)}_\omega $. 

Let $g_t=\begin{pmatrix}e^{t}&0\\0&e^{-t}\end{pmatrix}$ denote the diagonal subgroup of $\SL(2,\bR)$, whose action on the orbit closure $X$ yields the Teichm\"uller flow.
Let $\lambda_1\geq \lambda_2 \geq \dots \geq \lambda_h$ denote the non-negative Lyapunov exponents of the Kontsevich--Zorich cocycle on a symplectic strongly irreducible subbundle
$\textbf{H}$ of dimension $2h\in \{2,\dots, 2g\}$. Since the cocycle is symplectic on $\textbf{H}$, the top $h$ exponents determine the entire Lyapunov spectrum.

For $\omega\in X$, let $\textbf{v}_\omega$ in $\bP (\textbf{H}^{(k)}_\omega)$  be any $k$-dimensional exterior vector (of dimension $k\leq h$) in $\textbf{H}_\omega$. Let $\nu^*$ denote the probability measure on $\bP (\textbf{H}^{(k)})$ which projects to the $\SL(2,\R)$-invariant
measure $\nu$ on $X$ with conditional measures  equal to the
Lebesgue measure on $\bP (\textbf{H}^{(k)}_\omega)$ for $\nu$-almost all $\omega \in X$. We note that the family \((g_t)_{\ast}\nu^{\ast}\) of push-forward measures of \(\nu^{\ast}\) converges to $\hat \nu$, in the weak\(^\ast\) topology under the action of \(g_t\), as $t\to\infty$.

Let  $\|\cdot\|^{(k)}_{\pi(\cdot)}$ denote the $\SO(2,\bR)$-invariant Hodge norm on 
$\bP (\textbf{H}^{(k)})$ (see section \ref{hodgenorm} for the definition) and define 
$\sigma_k : \SL(2,\bR) \times \bP (\textbf{H}^{(k)}) \to \bR$ by \[\sigma_k(g,\textbf{v}) = \log \frac{\|g \textbf{v}\|^{(k)}_{g\pi(\textbf{v})}}{\|\textbf{v}\|^{(k)}_{\pi(\textbf{v})}}.\]

For $\nu^*\text{-a.e.}$~ $\text{\bf{v}} = (\omega,\textbf{v}_\omega)$, it is a consequence of the multiplicative ergodic theorem that \[\lim_{T\to\infty} \frac{\sigma_k(g_T, \textbf{v})}{T}  = \sum_{i=1}^k \lambda_i\,.\] 

Our main result is the following: 
\begin{thm}\label{dclt}
Let $\text{\bf{H}}$ be a strongly irreducible, symplectic, $\SL(2,\R)$-invariant subbundle, which is symplectic orthogonal to the tautological subbundle. If $\lambda_k > \lambda_{k+1}$, then there exists a real number $V^{(k)}_{g_\infty} \geq 0$ such that
\begin{align*}
\lim_{T\to\infty}\nu^*\left(\left\{\text{\bf{v}}\in \bP(\text{\bf{H}}^{(k)}) ~:~ a \leq \frac{1}{\sqrt{T}} \left(\sigma_k(g_{T}, \text{\bf{v}})- T(\sum_{i=1}^k \lambda_i) \right) \leq b\right\}\right) 
\\= \frac{1}{\sqrt{2\pi V^{(k)}_{g_\infty}}}\int_a^b\exp(-x^2/ V^{(k)}_{g_\infty})dx.
\end{align*}
Moreover, if the Lyapunov spectrum is simple, then $V^{(1)}_{g_\infty} > 0$.
\end{thm}
\begin{rem}
The statement also holds in the event that $V^{(k)}_{g_\infty} = 0$, and in that case the resulting distribution would  be  a delta distribution. The positivity of the variance holds for $2$-dimensional subbundles with strictly positive top Lyapunov exponent (for instance on the symplectic orthogonal of the tautological subbundle in genus $2$ for any $\SL(2,\R)$-invariant measure), as in this case the simplicity condition on the top  exponent is trivially satisfied.
\end{rem}
\begin{rem}
The simplicity of the Lyapunov spectrum is established for the canonical Masur-Veech measures on strata by A.~Avila and M.~Viana in \cite{AV07}, and we remark that, in genus $2$, this is established for all ergodic $\SL(2,\bR)$-invariant probability measures by M.~Bainbridge in \cite{B07}.
\end{rem}
\begin{rem}
The assumption that $\textbf{H}$ is symplectic orthogonal to the tautological subbundle precludes the equality $2h=2g$, which in turn precludes the equality $k = g$. It also follows by the spectral gap property of the Kontsevich-Zorich cocycle that for any $g_t$-invariant and ergodic probability measure, $\lambda_1<1$ \cite{For02} (see also \cite[Corollary 2.2]{FMZ12}).
\end{rem}

Since in Theorem \ref{dclt} we randomize \emph{both} the Abelian differential $\omega$ and vector $\bfv_\omega \in \bP(\bfH_\omega^{(k)})$ with respect to the measure $\nu^*$, it is natural to ask if our results also hold for future-Oseledets-generic sections of $\bP(\bfH^{(k)})$ (see Definition \ref{fut-gen}). In particular, for the deterministic cocycle, we derive in Section \ref{sectionsclt}:

\begin{cor}
Under the hypothesis of Theorem \ref{dclt}, there exists a real number $V^{(k)}_{g_\infty} \geq 0$ such that for any future-Oseledets-generic section $\bfv=(\omega,\bfv_{\omega})$ of $\bP(\bfH^{(k)})$, we have
\begin{align*}
\lim_{T\to\infty} \nu\left(\left\{\omega\in X ~:~ a \leq \frac{1}{\sqrt{T}} \left(\sigma_k(g_{T}, \bfv_\omega)- T(\sum_{i=1}^k \lambda_i) \right) \leq b\right\}\right) 
\\= \frac{1}{\sqrt{2\pi V^{(k)}_{g_\infty}}}\int_a^b\exp(-x^2/ V^{(k)}_{g_\infty})dx.
\end{align*}
Moreover, if the Lyapunov spectrum is simple, then $V^{(1)}_{g_\infty} > 0$.
\end{cor}

To prove Theorem \ref{dclt}, we will first work with the hyperbolic Brownian motion, which is the diffusion process generated by the foliated hyperbolic Laplacian. Let $\rho$ be a (foliated) hyperbolic Brownian motion trajectory starting at a (generic) basepoint $\omega \in X$, defined almost everywhere with respect to a probability measure $\bP_\omega$ on the space of such trajectories $W_\omega$. This process is in fact defined on $X_* = \SO(2,\bR)\backslash X$. The space $X_*$ gives rise to a fibered space $X^{W}_*$ whose fiber over each point $\omega$ in $X_*$ is $W_\omega$, and which also supports a measure $\nu_{\bP} \defeq \nu \otimes \bP$, whose projection on $X_*$ is $\nu$  and whose conditional measure over a point $\omega$ is $\bP_\omega$. We can thus similarly define the product $W$-Hodge bundle $\bP^W(\text{\bf{H}}^{(k)})$, whose fiber over each point $(\omega,\rho)$ in $X^{W}_*$ is $\textbf{H}^{(k)}_\omega$. A pair $(\rho,\text{\bf{v}}) \in \bP^W(\textbf{H}^{(k)})$ is thus defined to be the lift of the path $\rho$ (starting at $\omega$) to $\bP^W(\textbf{H}^{(k)})$, obtained by parallel transport with respect to the Gauss-Manin connection.  This in turn would also give rise to a measure $\nu^*_{\bP}\defeq \nu^* \otimes \bP$, whose projection on $\bP(\bfH^{(k)})$ is $\nu^*$ and whose conditional measure over a point $\text{\bf{v}}$ is $\bP_\omega$. We therefore also have

\begin{thm}\label{rclt}
Let $\text{\bf{H}}$ be a strongly irreducible, symplectic, $\SL(2,\R)$-invariant sub-bundle, which is symplectic orthogonal to the tautological sub-bundle. If $\lambda_k > \lambda_{k+1}$, then there exists a real number $V^{(k)}_{\rho_\infty} > 0$ such that
\begin{align*}
\lim_{T\to\infty}\nu^*_{\bP}\left(\left\{(\rho,\text{\bf{v}}) \in \bP^W(\text{\bf{H}}^{(k)}) : a \leq \frac{1}{\sqrt{T}} \left(\sigma_k(\rho_{T}, \text{\bf{v}})- T(\sum_{i=1}^k\lambda_{i})\right) \leq b\right\}\right) 
\\ = \frac{1}{\sqrt{2\pi V^{(k)}_{\rho_\infty}}}\int_a^b\exp(-x^2/V^{(k)}_{\rho_\infty})dx.\end{align*}
\end{thm}

\begin{rem}
Observe that for $g=2$, the symplectic orthogonal bundle to the tautological bundle has dimension $2$, hence it is strongly irreducible. Our two results reduce to ones that concern the second Lyapunov exponent of the Kontsevich-Zorich cocycle on the full Hodge bundle.  
\end{rem}

In addition to the Hodge theoretic techniques that we employ, some ingredients of our proof include 

\begin{itemize}
\item Exponential mixing of the \Tm~geodesic flow (due to Avila-Gou\"ezel-Yoccoz \cite{AGY06} for strata and Avila-Gou\"ezel \cite{AG13} for all $\SL(2,\R)$-invariant orbifolds) to derive the existence of a unique zero-average solution of a Poisson equation (see 
Appendix~\ref{poisson});
\item elementary stochastic calculus to extract and control the necessary oscillations;
\item and an asymptotic estimate due to Ancona \cite{Anc90} to relate the geodesic flow with the Brownian motion.
\end{itemize}

\section*{Acknowledgements} This paper is an outgrowth of the PhD dissertation of the first-named author, written under the direction of the second-named author, at the University of Maryland, College Park. The authors thank the anonymous referees for their careful reading of the paper, and for many helpful comments and suggestions that led to a substantial improvement in the exposition. The first-named author is grateful to Jon Chaika, Dima Dolgopyat, and Bassam Fayad for their interest in this project and for insightful and enlightening conversations. 
The first-named author acknowledges support from the NSF grants DMS 1600687, DMS 1956049, DMS 2101464, and the the Deutsche Forschungsgemeinschaft (DFG) Grant No. 422642921 (Emmy Noether group “Microlocal Methods for Hyperbolic Dynamics”). The second-named author acknowledges support from the NSF grants DMS 1600687 and DMS 2154208.

  \section{Preliminaries}
  \subsection{Translation surfaces}
Let $S$ be a Riemann surface of genus $g\geq 2$, and $\omega$ a holomorphic $1$-form on $S$. 
The pair $(S,\omega)$ is called a {\it translation surface}, since $\omega$ induces an atlas whose coordinate changes are translations on $\C\equiv \R^2$.
In other terms, $\omega$ gives a flat metric with finitely many conical singularities and trivial holonomy on $S$, and
the zero set of $\omega$ characterizes the singularity set of the conical metric. The area of a translation surface is given by $\int_S \omega \wedge \overline{\omega}$. We will refer to the pair $(S,\omega)$ as just $\omega$.

\subsection{Moduli Space}
Let $ \cT\cH_g$ be the \Tm~space of unit-area translation surfaces of genus $g \geq 2$, and let $ \cH_g =  \cT\cH_g / \Mod_g$ be the corresponding moduli space, where $\Mod_g$ denotes the mapping class group. The space $\cH_g$ is partitioned into strata  $\cH_{\kappa}$, which consist of all unit-area translation surfaces whose conical singularities have total angles $2\pi(1+\kappa_1), \dots, 2\pi(1+\kappa_s)$, as $\kappa=(\kappa_1, \dots, \kappa_s)$ varies over multi-indices with $\sum \kappa_i = 2g-2$.

Local period coordinates on each stratum are defined by the map which takes every holomorphic $1$-form  $\omega$ to its cohomology class $[\omega]$ in $H^1(S, \Sigma_\omega, \C)$, relative to the set $\Sigma_\omega$ of its zeros.
The set of all period coordinate maps defines an affine structure on each stratum, since all changes of coordinates are given by affine maps. 

\subsection{SL(2,R) action}
There is a natural action of $\SL(2,\bR)$ on the space of all translation surfaces which descends to their Teichm\"uller and moduli spaces. It is proved in \cite{EM18, EMM15} that, for any $\omega \in \cH(\kappa)$, the closure $X$ of $\SL(2,\bR) \cdot \omega$ is an affine invariant suborbifold, and supports a unique ergodic $\SL(2,\bR)$-invariant probability measure $\nu$ in the Lebesgue measure class, given by the normalized Lebesgue measure in period coordinates.

\subsection{Kontsevich-Zorich cocycle}
Let $\widehat{\textbf{H}}_\kappa(S, \bR) = \cT\cH_\kappa \times H^1(S,\mathbb{R})$, and for every  $g\in \SL(2,\bR)$, let $\widehat{g} : \widehat{\textbf{H}}_\kappa(S, \bR) \to \widehat{\textbf{H}}_\kappa(S, \bR)$ be the trivial cocycle map defined as $$\widehat{g}(\omega,c) = (g\omega, c)\,, \quad \text{ for }\omega \in  \cT\cH_\kappa \text{ and } c \in H^1(S,\mathbb{R})\,,.$$ 
The absolute (real) Hodge bundle is given by $\textbf{H}_\kappa(S, \bR) = \widehat{\textbf{H}}_\kappa(S, \bR) / \Mod_g$ and the Kontsevich-Zorich cocycle $g$ is the projection of $\widehat{g}$ to $\textbf{H}_\kappa(S, \bR)$.

\subsection{Hodge inner product and the second fundamental form}\label{hodgenorm} Given two holomorphic $1$-forms $\omega_1,\omega_2$ in $\Omega(S)$, where $\Omega(S)$ is the vector space of holomorphic $1$-forms on $S$, the Hodge inner product is given by the formula \[\langle\omega_1,\omega_2 \rangle := \frac{i}{2}\int_S \omega_1 \wedge \overline{\omega_2}\]
Moreover, the Hodge representation theorem implies that for any given cohomology class $c \in H^1(S,\bR)$, there is a unique holomorphic $1$-form $h(c) \in \Omega(S)$, such that $c = [\Re~ h(c)]$ (cf. \cite{FMZ12}). The Hodge inner product for two real cohomology classes $c_1,c_2 \in H^1(S,\bR)$
is defined as \[A_\omega(c_1,c_2 ) := \langle h(c_1),h(c_2) \rangle = \frac{i}{2}\int_S h(c_1) \wedge \overline{h(c_2)}  \,.\]\,
The second fundamental form $B_\omega$ (of the Gauss-Manin connection with respect to the Chern connection for the holomorphic structure of the Hodge filtration) is defined as \[B_\omega(c_1,c_2) := \frac{i}{2}\int_S \frac{h(c_1)h(c_2)}{\omega^2} \omega \wedge \overline{\omega}.\] \,
Let $H_\omega$ denote the curvature operator of the second fundamental form. 

\begin{rem}\label{b-vanishing}
It is known that $B_\omega$ vanishes identically in the symplectic orthogonal of the tautological subbundle on only two orbit closures, namely the \emph{Eierlegende Wollmilchsau} and \emph{Ornithorynque}, and this follows from the works \cite{Aul06, EKZ14, Mol05,AN19}. By a result of S. Filip \cite{Fil17}, the rank of the second fundamental form $B$ is a lower bound for the number of strictly positive Lyapunov exponents of the Kontsevich--Zorich cocycle.
\end{rem}
%\sloppy
In the following $\textbf{H}$ will denote a strongly irreducible, symplectic, $\SL(2,\R)$-invariant subbundle, which is symplectic orthogonal to the tautological bundle. For any isotropic $k$-dimensional exterior vector $\text{\bf{c}}_\omega$ in $\bfH^{(k)}_\omega$, it also follows by \cite{For02} (see also \cite[Corollary 2.2]{FMZ12}) that  \begin{align} \left|\frac{d}{dt} \sigma_k(g_t, \text{\bf{c}}_\omega)\right| < k\label{lipschitzprop}\end{align} 

For $h\in \{1,\dots, g-1\}$ and $\textbf{H}$ of dimension $2h$, let $\{c_1, c_2,\ldots,c_{h}\}$ be a Hodge-orthonormal basis of $\textbf{H}^{1,0} \subset H^1(S, \bC)$, and let $A^{(h)}_\omega$  (resp., $B^{(h)}_\omega$) be the corresponding representation matrix of the Hodge inner product $A_\omega$ (resp., of the second fundamental form $B_\omega$). Let $H^{(h)}_\omega = B^{(h)}_\omega \bar B^{(h)}_\omega$  be the matrix of the curvature operator, which is Hermitian non-negative, since $B^{(h)}_\omega$ is symmetric. The eigenvalues of $B^{(h)}_\omega$ are denoted by $\Lambda_i(\omega)$, where $ |\Lambda_1| > |\Lambda_2| \geq \cdots \geq |\Lambda_h| \geq 0$. Moreover, the norm squared of these eigenvalues, $|\Lambda_i(\omega)|^ 2$, are the eigenvalues of the curvature matrix $H^{(h)}_\omega$, which are continuous, bounded functions on $\cH_g$ (cf. \cite{FMZ12}, Lemma 2.3). 

For any $k$-dimensional exterior vector $\text{\bf{v}} \in \bP(\textbf{H}^{(k)})$, let 
$$
\{c_1, c_2,\ldots, c_k, c_{k+1}, \ldots, c_{h}\} \subset 
\textbf{H}
$$ 
be an ordered orthonormal basis such that 
$\{c_1, c_2,\ldots, c_k\}$ is a basis of $\text{\bf{v}}$.
We let $A^{(k)}_\omega(\text{\bf{v}})$  (resp., $B^{(k)}_\omega(\text{\bf{v}})$) be the corresponding representation matrix of the Hodge inner product $A_\omega$ (resp., of the second fundamental form $B_\omega$)
restricted to $\text{\bf{v}}$ with respect to the basis 
$\{c_1, c_2,\ldots, c_k\}$. We let $H^{(k)}_\omega (\text{\bf{v}})$
be the representation matrix of the restriction of the curvature operator $H_\omega$ to $\text{\bf{v}}$ with respect to the basis $\{c_1, c_2,\ldots, c_k\}$.

\subsection{Foliated Hyperbolic Laplacian}
The space $\cH_g$, is foliated by the orbits of the $\SL(2,\bR)$-action, whose leaves are isometric to the unit cotangent bundle of the Poincaré disk $\bD$ or of a finite volume hyperbolic surface. Unless $X$ is a closed $\SL(2,\R)$-orbit (the orbit of a Veech surface), almost all leaves are isometric to the unit cotangent bundle of the Poincaré disk. 

For almost every $\omega \in \cH_g$, the \Tm~ disk $L_\omega \defeq \SL(2,\bR)/\SO(2,\bR) \cdot \omega$ is isometric to $\bD$ (or a finite area hyperbolic surface $\bD/ V(\omega)$ defined by a lattice $V(\omega)<\SL(2,\R)$ called the Veech group), and so is endowed with the (foliated) hyperbolic gradient $\gradient_{L_\omega}$ and hyperbolic Laplacian $\laplacian_{L_\omega}$. Let $ r_\theta=\begin{pmatrix}  \cos(\theta/2) &-\sin(\theta/2)\\ \sin(\theta/2)& \cos(\theta/2) \end{pmatrix}$.

\begin{rem}\label{embedding}
Observe that for $\omega \in X$, the \Tm~disk $L_\omega$ is identified with $\bD$ (or $\bD/V(\omega)$) via the map $(t,\theta) \mapsto \SO(2,\bR) \cdot g_{t} r_{\theta} \omega$.
\end{rem}

Now suppose that $f: X\rightarrow\R$ is an $\SO(2,\bR)$-invariant $C^\infty$-function in the direction of the leaf. For $\omega \in X$ and for $L_\omega$ the \Tm~ disk passing through $\omega$, we define $\laplacian f(\omega) \defeq \laplacian_{L_\omega} f|_{L_\omega}(\omega)$, where $f|_{L_\omega}$ is the restriction of $f$ to $L_\omega$. We also define the leaf-wise gradient similarly. 

Observe that the Hodge inner product $A_{\omega}(\cdot,\cdot)$ is invariant under the action of $\SO(2,\bR)$, and so defines a real-analytic function on the \Tm~ disk. In the sequel, we will only work in a given \Tm~ disk, so the norm will read $(\cdot,\cdot)_z$ for a complex parameter $z \in \bD$. For any $k$-dimensional exterior vector  $\text{\bf{v}}=(\omega,\text{\bf{v}}_{\omega})$ in the symplectic orthogonal of the tautological subbundle (with the origin $z=0$ corresponding to $\omega$ as in \eqref{embedding}), define \[\sigma_k(z,\text{\bf{v}}) := \log |\det A^{(k)}_{z} (\text{\bf{v}})|^{1/2},\]
where $A^{(k)}_{z}(\text{\bf{v}}) = A_z(\text{\bf{v}}_i,\text{\bf{v}}_j)$  and $\{\text{\bf{v}}_i\}$ is an ordered basis of $\text{\bf{v}}$.
\begin{rem}In fact, this is an abuse of notation since we originally lifted elements of $\SL(2,\bR)$ to the Hodge bundle. This is not an issue since the Hodge norm is $\SO(2,\bR)$-invariant.
\end{rem}

We recall the following fundamental fact \begin{thm}\cite{For02, FMZ12}\label{det} For every $1\leq k\leq h$
there exist smooth functions $\Phi_k:\bP(\textbf{H}^{(k)}) \to [0,k]$ and $\Psi_k :\bP(\textbf{H}^{(k)}) \to D(0,k) \subset \bC$  such that the following holds.
For any $k$-dimensional exterior vector $\text{\bf{v}} \in \bP(\textbf{H}^{(k)})$, we have the following identities: 
\begin{align}
\laplacian_{L_\omega} \sigma_k(z,\text{\bf{v}}) = 2\Phi_k (z, \text{\bf{v}})\quad \text{ and }  \quad
\gradient_{L_\omega} \sigma_k(z,\text{\bf{v}}) = \Psi_k (z, \text{\bf{v}})\,.\label{laplacianformula}
\end{align}
In the particular case that $k=h$, there exist functions
$\Lambda_i: X \to D(0,1)$ for all $i\in \{1, \dots, h\}$  such that
$$
\laplacian_{L_\omega} \sigma_h(z,\text{\bf{v}}) = 2 \sum_{i=1}^{h} |\Lambda_i(z)|^2 \quad \text{ and } \quad 
\gradient_{L_\omega} \sigma_h(z,\text{\bf{v}}) =\sum_{i=1}^{h} \Lambda_i(z)\,.
$$
In this case the Laplacian and the gradient are independent of the choice of a maximal isotropic  (Lagrangian) subspace $\text{\bf {v}}\in \bP(\textbf{H}^{(h)})$.  

Moreover, for all $k\in \{1, \dots, h\}$, for any $\SL(2,\R)$-invariant measure $\nu$ on $\cH_\kappa$, under the condition that $\lambda_k > \lambda_{k+1}$, which implies that the unstable Oseledets isotropic $k$-dimensional distribution $E^+_k$ is well-defined $\nu$-almost everywhere, and for $\hat \nu$ almost all $\text{\bf {v}}\in \bP(\textbf{H}^{(h)})$,  we have that
\begin{align*}
\lim_{T\to\infty} \frac{1}{T}\int_0^T \laplacian_{L_\omega} \sigma_k\big(g_t,\text{\bf{v}}\big) \, dt = \int_X 2\Phi_k 
(\omega, E^+_k(\omega)) d\nu = 2\sum_{i=1}^k \lambda_i\,.
\end{align*}
\end{thm}
\begin{rem}
The functions $\Phi_k$ and $\Psi_k$ can be written as follows.
Let $B^{(k)}_z({\bf v})$ and $H^{(k)}_z( {\bf v})$ denote the
restrictions of the second fundamental form and of the curvature to the $k$-dimensional exterior vector ${\bf v} \in \bP(\textbf{H}^{(k)})$.  By definition $B^{(k)}$ and $H^{(k)}$ are  functions on $\bP(\textbf{H}^{(k)})$  with values in the subspace of complex symmetric $k\times k$ matrices and non-negative Hermitian $k\times k$ matrices. The following formulas hold, for all $(\omega,{\bf v})\in \bP(\textbf{H}^{(k)})$:
\begin{equation}
\label{eq:PhiPsi}
\begin{aligned}
\Phi_k (\omega, \text{\bf{v}})&= 2\text{tr} (H^{(k)}_\omega(\text{\bf{v}}))-
\text{tr} \left(B^{(k)}_\omega(\text{\bf{v}}) \bar B^{(k)}_\omega (\text{\bf{v}})\right) \,;\\
\Psi_k (\omega, \text{\bf{v}})&=  \text{tr}(B^{(k)}_\omega(\text{\bf{v}}))  \,.
\end{aligned}
\end{equation}
\end{rem}

An important property of the function $\Psi_k$, which is relevant in the proof of positivity of the variance (for $k=1$)
in Section \ref{posvar:det}, is that for all $\omega \in X$ the set of its critical point is a subset of the level set 
$$\{\textbf{v} \in \bP_\omega(\textbf{H}^{(k)}) \vert \Psi_k(\omega, \textbf{v})=0\}\,.$$
In fact, the following linear algebra result holds:
\begin{lem} 
\label{lemma:crit}
For every $\omega \in X$, the set of critical points of the function $\Psi_k(\omega, \cdot)$ equals the  set 
$$
\{ \text{\bf{v}} \in \bP(\textbf{H}_\omega^{(k)}) \vert  B^{(k)}_\omega(\text{\bf{v}})=0\} \subset \{\text{\bf{v}} \in \bP(\textbf{H}_\omega^{(k)}) \vert \Psi_k(\omega, \text{\bf{v}})=0\}\,.
$$
In addition,  for all $\omega \in X$ and
$\text{\bf{v}} \in \bP_\omega (\textbf{H}^{(k)})$, we have 
$$
\Vert D_{ \text{\bf{v}} }\Psi_k(\omega, \text{\bf{v}}) \Vert\geq   2  \vert \Psi_k(\omega, \text{\bf{v}})\vert/k\,.
$$

\end{lem}
\begin{proof}
By its definition, the function $\Psi_k$ is computed as follows. 

Let $\{v_1, \dots, v_k\}$ denote any Hodge orthonormal
basis of the $k$-dimensional isotropic subspace $\text{\bf{v}} \in \bP_\omega(\textbf{H}^{(k)}) $, then
$$
\Psi_k(\omega, \text{\bf{v}}) = \sum_{i=1}^k B_\omega(v_i, v_i)\,.
$$
It can be seen that the above expression depends only on the isotropic subspace $\text{\bf{v}}$ and not on its orthonormal basis.  By fixing a orthonormal basis $\{w_1, \dots, w_h\}$
of a maximal isotropic (Lagrangian) subspace, the space of all orthonormal bases is in bijective correspondence with the group
of complex unitary matrices.  

Let then $U =(U_{ij})$ denote a complex unitary matrix (such that $UU^*=I$) and for all $i\in \{1, \dots, h\}$  let 
$$
w_i = \sum_{a=1}^h U_{ia} v_a \,.
$$
We then have
$$
\Psi_k(\omega, \text{\bf{w}}) = 
\sum_{i=1}^k B_\omega(w_i, w_i) = \sum_{i=1}^k
\sum_{a,b=1}^h U_{ia} U_{ib} B(v_a, v_b)\,.
$$
The tangent space at the identity of the unitary group is the vector space of anti-Hermitian matrices, that is, the matrices
$T$ such that $T^* = -T$. By differentiating the above formula
with respect to $U$ along $T$ we have
\begin{equation}
\label{eq:Psi_der}
(D_U\Psi_k)(\omega, \text{\bf{v}}, T)= 2 
\sum_{a,b=1}^k T_{ab}  B(v_a, v_b)\,.
\end{equation}
It follows that, if $\text{\bf{v}}$ is a critical point,
for $T$ such that $T_{\alpha\beta}=0$ for all $(\alpha,\beta) \not\in \{(a,b), (b,a)\}$  we have
$$
(T_{ab} + T_{ba}) B(v_a, v_b) =0 
$$
which if $T_{ab} + T_{ba}= T_{ab} - \overline{T_{ab}}\not=0$,
implies $B(v_a, v_b)=0$.

\smallskip
\noindent Finally, the stated lower bound follows
from formula \eqref{eq:Psi_der} since 
$$
\Vert D_{ \text{\bf{v}} }\Psi_k(\omega, \text{\bf{v}}) \Vert\geq 2 (\max_{i\in \{1, \dots, k\}} \vert B_\omega(v_i,v_i)\vert) \geq 2 \vert \Psi_k (\omega, \text{\bf{v}}) \vert /k\,.
$$
The proof is therefore complete.
\end{proof}

\subsection{Harmonic measures}
A probability measure $\nu$ on $X_*:=\SO(2,\bR)\backslash X$ is called harmonic if, for all  functions $f :  X_* \to \bR$ in the domain of the foliated
Laplacian $\Delta$ on $X_*$, we have
\[\int_{X_*} \Delta f\, d\nu= \int_{X_*} \Delta_{L_\omega} f|_{L_\omega}(\omega)\, d\nu(\omega)=0.\]
 Such a measure is also ergodic if $X_*$ cannot be partitioned into two union of leaves, each of which having positive $\nu$ measure. We refer the reader to the interesting paper of Lucy Garnett \cite{Gar83} for details and for an ergodic theorem for such measures. It is also a fact, due to Bakhtin-Martinez \cite{BM08}, that harmonic measures on $\SO(2,\bR)\backslash X$ are in one-to-one correspondence with $P$-invariant measures on $X$. This is closely related to a classical fact due to Furstenberg \cite{Fur2,Fur1} that $P$-invariant measures are in one-to-one correspondence with (admissible) stationary measures, and that harmonic measures are stationary. In the case of $\SL(2,\bR)$, these three notions are therefore closely related. 

 It follows from \cite{BEW20} that there is a unique harmonic
 measure $\hat \nu_*$ on $\SO(2, \R) \backslash \bP(\bf H^{(k)})$, which projects onto the quotient measure of the $\SL(2,\R)$-invariant measure $\nu$ on $X_*= SO(2, \R)\backslash X$. The harmonic measure $\hat \nu_*$ is the projection of the $P$-invariant measure $\hat \nu$ on 
 $\bP(\bf H^{(k)})$ to $\SO(2, \R) \backslash \bP(\bf H^{(k)})$.
 Let then $W$ the space of trajectories of the foliated Brownian motion on $X_*$, endowed with the Wiener measure $\bP$, and let $\hat \nu_{\bP}\defeq \hat \nu_* \otimes \bP$ denote the probability measure on the vector bundle  $\bP^W (\bfH^{(k)})$ over the fiber bundle $X_* ^W$, whose projection on $\bP(\bfH^{(k)})$ is the harmonic measure $\hat\nu_*$ and whose conditional measure over a point $\text{\bf{v}}$ is $\bP_\omega$.

 \subsection{Hyperbolic Brownian Motion}
Following the normalization used in \cite{For02} (which is a standard normalization, see also \cite{Hel00}), for $z = re^{i\theta}$ with $\theta \in [0,2\pi]$, write \begin{align} t \defeq \frac{1}{2} \log \frac{1+r}{1-r}.\end{align}  

Since the Hodge norm is $\SO(2,\bR)$-invariant, it suffices to study the diffusion process generated by $\frac{1}{2}\laplacian_{L_\omega}$, where the leaf-wise hyperbolic Laplacian in geodesic polar coordinates is given by \begin{align}\laplacian_{L_\omega} = \frac{\partial^2}{\partial t^2} + 2\coth(2t)\frac{\partial}{\partial t}+\frac{4}{\sinh^2(2t)}\frac{\partial^2}{\partial \theta^2}.\end{align} Moreover, let $(W_\omega^{(i)},\bP_\omega^{(i)})$, $i =1, 2$, be two copies of the space of Brownian trajectories $C(\bR^{+},\bR)$ starting at the origin (with the origin corresponding to a random point $\omega$), together with the standard Wiener measure, and such that $W_\omega^{(1)}$ and $W_\omega^{(2)}$ are independent. Set $W_\omega = W_\omega^{(1)}\times W_\omega^{(2)}$ and $\bP_\omega = \bP_\omega^{(1)} \times \bP_\omega^{(2)}$. The hyperbolic Brownian motion is the diffusion process $\rho_s = (t(s),\theta (s))$ generated by the (leaf-wise) hyperbolic Laplacian. It follows by Ito's formula \cite[Theorem IV.3.3]{RY} that the generator determines the trajectories of the diffusion process $\rho_s$ which are solutions of the following stochastic differential equations
\begin{align}
dt(s) &= dW_s^{(1)} + \coth(2t(s)) ds\label{tsdef} \\
d\theta(s) &= \frac{2}{\sinh(2t(s))} dW_s^{(2)} 
\end{align}
with $t(0) = 0$ and $\theta(0)$  being uniformly distributed on $S^1$.

In addition, for an $\SO(2,\bR)$-invariant function $f : \bP(\bfH) \to \bR$, where $f$ is of class $C^2$ along $\SL(2,\bR)$ orbits, Ito's formula gives 
\begin{align*}
&f(\rho_T,\bfv) - f(\rho_0,\bfv) = \int_0^T \left(\frac{\partial}{\partial t} f(\rho_s,\bfv),\frac{2}{\sinh(2t(s))} \frac{\partial}{\partial \theta} f(\rho_s,\bfv)\right) \cdot \left(dW^{(1)}_s,dW^{(2)}_s\right) \\ &+ \int_0^T \left(\frac{1}{2}\frac{\partial^2 }{\partial t^2} f(\rho_s,\bfv) + \frac{1}{2}2\coth(2t(s))\frac{\partial}{\partial t} f(\rho_s,\bfv)+\frac{1}{2}\frac{4}{\sinh^2(2t(s))}\frac{\partial^2 }{\partial \theta^2} f(\rho_s,\bfv) \right)ds \\
&= \int_0^T \gradient_{L_\omega} f(\rho_s,\bfv) \cdot (dW^{(1)}_s,dW^{(2)}_s) + \frac{1}{2}\int_0^T \laplacian_{L_\omega} f(\rho_s,\bfv) ds\,.\label{regularito}
\end{align*}

Finally, we note that the foliated heat semigroup $D_t$ is given as follows 

\begin{align*}
D_s f(x,\bfv)&\defeq  \frac{1}{2\pi}\int_0^{2\pi} \int_0^\infty f(z,\bfv) p_{\omega}(t,s) \sinh(t) d t\, d\theta\,,
\end{align*}
where $p_{\omega}(t,s)$ is the (foliated) hyperbolic heat kernel at time $s$; in other words, for $x,y \in L_\omega$, this is the (rotationally invariant) transition probability kernel $p_{\omega} (x,y;s)$, with $d_{\bD}(x,y) = t$.

\section{Proofs of Main Theorems}
\subsection{Proof of the Distributional Convergence in Theorem \ref{rclt}}\label{randomproof} Recall that $\rho_s$ is the diffusion process generated by the foliated hyperbolic Laplacian. We are interested in studying the term \begin{align}\frac{1}{\sqrt{T}} (\sigma_k(\rho_T,\text{\bf{v}}) - T\sum_{i=1}^k\lambda_i).\end{align} Set $\lambda_{(k)} = \sum_{i=1}^k \lambda_i$. By applying Ito's formula, we obtain, 
 
\begin{align}
\frac{1}{\sqrt{T}} (\sigma_k(\rho_T,\text{\bf{v}}) - T\lambda_{(k)}) &= \frac{\sigma_k(\rho_0,\text{\bf{v}})}{\sqrt{T}}  +\frac{1}{\sqrt{T}}\int_0^T \gradient_{L_\omega}\sigma_k(\rho_s,\text{\bf{v}}) \cdot (dW^{(1)}_s,dW^{(2)}_s) \\   &+ \frac{1}{2\sqrt{T}}\int_0^T (\laplacian_{L_\omega} \sigma_k(\rho_s,\text{\bf{v}})-2\lambda_{(k)})ds
\end{align}

Let $W^{2,2}(\bP(\bfH^{(k)}), \nu)$ denote the (foliated) Sobolev space of functions which belong to $L^2(\bP(\bfH^{(k)}), \nu)$ together with all their derivatives up to second order, in all directions tangent to $\SL(2, \R)$ orbits: for all  $\cV, \cW \in \mathfrak{sl}(2, \R)$, 
$$
f \in W^{2,2}(\bP(\bfH^{(k)}), \nu) \Longleftrightarrow f, \,\cV f, \, \cV \cW f \in L^2(\bP(\bfH^{(k)}), \nu)\,.
$$

 It follows  by Lemma \ref{lemma:Poisson}  that the equation 
 \begin{equation}
  \label{poissoneq}
    \laplacian_{L_\omega} u(\omega,\bfv)  =  \laplacian_{L_\omega} \sigma_k(\omega,\bfv) - 2\lambda_{(k)}   \,, \quad \text{ for } (\omega,\bfv_\omega)\in \bP(\bfH^{(k)})\,,
 \end{equation}
has an $\SO(2,\bR)$-invariant solution $u^{(k)} \in W^{2,2}(\bP(\bfH^{(k)}), \nu)$, the space of functions with all $\mathfrak{sl}(2, \R)$-derivatives up to second order in $L^2(\bP(\bfH^{(k)}),\nu)$. 

In fact, since the functions $\laplacian_{L_\omega} \sigma = \Phi_k$ are smooth and bounded on $\bP(\bfH^{(k)})$ (see Theorem \ref{det}) and the foliated Laplacian $\laplacian_{L_\omega}$ is elliptic along the leaves, it follows that the functions $u^{(k)}$ are also smooth along the leaves.
 
 \smallskip
 \noindent For all $s>0$, let $\rho_s = (t(s), \theta(s))$ in geodesic polar coordinates. By Ito's formula we get,
\begin{align*}
\frac{1}{\sqrt{T}}(u(\rho_T,\bfv)-u(\rho_0,\bfv)) &=  \frac{1}{\sqrt{T}}\int_0^T \gradient_{L_\omega} u(\rho_s,\bfv) \cdot (dW^{(1)}_s,dW^{(2)}_s)  \\ & \qquad+ \frac{1}{2\sqrt{T}} \int_0^T  \laplacian_{L_\omega} u(\rho_s,  \bfv) ds 
\end{align*}

So, by substituting $\laplacian_{L_\omega} \sigma_k(\rho_s,\text{\bf{v}})-2\lambda_{(k)}$ for $\laplacian_{L_\omega} u(\rho_s, \bfv)$, we have that \begin{align*}
\frac{1}{2\sqrt{T}}&\left(\int_0^T (\laplacian_{L_\omega} \sigma_k(\rho_s,\text{\bf{v}})-2\lambda_{(k)})ds\right) \\
&=\frac{1}{\sqrt{T}}(u(\rho_T, \bfv)-u(\rho_0, \bfv))  \\& 
\qquad\qquad - \frac{1}{\sqrt{T}}\int_0^T \gradient_{L_\omega} u(\rho_s, \bfv) \cdot (dW^{(1)}_s,dW^{(2)}_s)\,. \end{align*}
Define 
\begin{align}
M_T = \int_0^T \gradient_{L_\omega} (\sigma_k(\rho_s,\text{\bf{v}})-u(\rho_s,\bfv) ) \cdot (dW^{(1)}_s,dW^{(2)}_s)\label{itointegralformula}
\end{align}
We then have 
\begin{align}
\frac{1}{\sqrt{T}} (\sigma_k(\rho_T,\text{\bf{v}}) - T\sum_{i=1}^k\lambda_i) &= \frac{1}{\sqrt{T}}(u(\rho_T,\bfv)-u(\rho_0,\bfv) +\sigma_k(\rho_0,\text{\bf{v}})) \\ &  + \frac{1}{\sqrt{T}}M_T\,. \label{M_T_identity}
\end{align}
Next, we study the quadratic variation $\langle M_T, M_T\rangle_{\hat{\nu}_\bP}$. \begin{rem} It is a fact that the covariance of two Ito integrals with respect to independent Brownian motions is zero, and this is a consequence of Ito isometry \cite[Theorem IV.2.2]{RY}. Indeed, for any \(F, G \in L^2(\bP(\text{\bf{H}}^{(k)}), \hat\nu)\), consider the following expectation: \[
\bE_{{\hat{\nu}_\bP}} \left[ \left( \int_0^t F(\rho_s) dW^{(1)}_s \right) \left( \int_0^t G(\rho_s) dW^{(2)}_s \right) \right]
\]
Applying Itô’s isometry, the expectation simplifies to:
\[
\bE_{{\hat{\nu}_\bP}}\left[\int_0^t F(\rho_s) G(\rho_s) d\langle W^{(1)}, W^{(2)} \rangle_s\right]
\]
Now, since \( W^{(1)} \) and \( W^{(2)} \) are independent Brownian motions, their quadratic covariation satisfies:
\[
d\langle W^{(1)}, W^{(2)} \rangle_s = 0
\]
for all \( s \). This directly implies:
\[
\bE_{{\hat{\nu}_\bP}}\left[\int_0^t F(s) G(s) d\langle W^{(1)}_s, W^{(2)}_s \rangle_s\right] = 0
\]
which proves that the stochastic integrals
\[
I_1 = \int_0^t F(\rho_s) dW^{(1)}_s, \quad I_2 = \int_0^t G(\rho_s) dW^{(2)}_s
\]
are uncorrelated; i.e. \(\bE_{{\hat{\nu}_\bP}}[I_1 I_2] = 0,\) as desired.\end{rem}

We are now ready to analyze $\langle M_T, M_T\rangle_{\hat{\nu}_\bP}$. Applying the previous remark, we derive:
\begin{align*}
\langle M_T, M_T\rangle_{\hat{\nu}_\bP} &= \E{\left(\int_0^T (\gradient_{L_\omega}\sigma_k(\rho_s,\text{\bf{v}})-\gradient_{L_\omega}u(\rho_s,\bfv) ) \cdot (dW^{(1)}_s,dW^{(2)}_s)\right)^2} \\ &= \E{\left(\int_0^T  (\frac{\partial \sigma_k}{\partial t}(\rho_s,\text{\bf{v}})-\frac{\partial u}{\partial t}(\rho_s,\bfv) ) dW^{(1)}_s\right)^2}\\  &+ \E{\left(\int_0^T \frac{2}{\sinh(2t(s))}(\frac{\partial \sigma_k}{\partial \theta}(\rho_s,\text{\bf{v}})-\frac{\partial u}{\partial \theta}(\rho_s,\bfv) ) dW^{(2)}_s\right)^2}  
\end{align*}

Applying Ito's isometry on the expectation of the square of the Ito integrals on the RHS yields
\begin{align*}
\langle M_T, M_T\rangle_{\hat{\nu}_\bP}  &= \E{\int_0^T \left( ( \frac{\partial \sigma_k}{\partial t}(\rho_s,\text{\bf{v}})- \frac{\partial u}{\partial t}(\rho_s,\bfv) )\right)^2 ds}\\  &+ \E{\int_0^T \left(\frac{2}{\sinh(2t(s))} (\frac{\partial \sigma_k}{\partial \theta}(\rho_s,\text{\bf{v}})- \frac{\partial u}{\partial \theta}(\rho_s,\bfv) )\right)^2 ds} \\
&= \E{\int_0^T | \gradient_{L_\omega}\sigma_k(\rho_s,\text{\bf{v}})-\gradient_{L_\omega}u(\rho_s,\text{\bf{v}})|^2ds}
\end{align*}

Observe that $|\gradient u| \in L^2(\bP(\bfH^{(k)}),\hat\nu)$ by Lemma \ref{lemma:Poisson}.
Therefore, by  Oseledets’ theorem, Fubini's theorem, and the dominated convergence theorem, we have the convergence with respect to the measure $\hat{\nu}$ on $\bP(\bfH^{(k)})$:

\begin{align}
V^{(k)}_{\rho_\infty} &\defeq \lim_{T\to\infty} \frac{1}{T}\E{\int_0^T | \gradient_{L_\omega}\sigma_k(\rho_s,{\bf{v}})- \gradient_{L_\omega}u(\rho_s,{\bf{v}})|^2ds} \\&= \int_{\bP(\bfH^{(k)})} |\Psi_k (\omega,{\bf{v}})-\gradient_{L_\omega}u(\omega,\bfv)|^2d \hat{\nu} \\
&=\int_{X} |\Psi_k (\omega,E^+_k(\omega)) -
\gradient_{L_\omega}u(\omega,E^+_k(\omega))|^2d \nu \,.\label{variance}
\end{align}
See also \cite[Corollary 5.5]{For02}. The above formula,  together with Theorem~\ref{flj:mclt}, implies that the random variables $M_T/ \sqrt{T}$, hence the random variables 
$(\sigma(\rho_T,{\bf v}) -\lambda_{(k)} T)/\sqrt{T}$,  converge in distribution to a centered Gaussian distribution of variance $V^{(k)}_{\rho_\infty}$. In fact, the convergence in distribution of $(\sigma(\rho_T,{\bf v}) -\lambda_{(k)} T)/\sqrt{T}$ can be deduced from that of $M_T/ \sqrt{T}$ 
as follows. Since the function $u \in L^2(\bP(\bfH),\nu)$, it follows that  
$$
\begin{aligned}
\bE_{{\hat{\nu}_\bP}} &\Big( \frac{\vert \sigma_k(\rho_T,{\bf v}) -\lambda_{(k)} T - M_T\vert^2}{T}\Big) \\ & \qquad= \frac{1}{T}  \bE_{{\hat{\nu}_\bP}}(\vert u(\rho_T, \bfv)- u(\rho_0, \bfv ) + \sigma_k(\rho_0, \bfv)\vert^2)
\end{aligned}
$$
converges to $0$. Thus the random variable $(\sigma(\rho_T,{\bf v}) -\lambda_{(k)} T -M_T)/\sqrt{T}$ converges to $0$ in square mean,
hence in distribution. \hfill\qedsymbol

 \subsection{Proof of the Distributional Convergence in Theorem \ref{dclt}}
 
Observe that $t(s) = d_{\bD}(0,\rho_s)$, and that it is rotationally invariant. We will need the following useful lemma: 

\begin{lem}\cite[Lemma VII.7.2.1]{FL12}\label{LFL}
For all $\omega \in X$, there exists an $\bP_\omega$-almost everywhere converging process $\eta_s$ such that $t(s) = W_s^{(1)} + s+\eta_s$.
\end{lem}
\begin{proof}
It is a classical fact that $t(s)\to\infty$ $\bP_\omega$-almost everywhere. This implies that $\lim_{s\to\infty} \coth (2t(s)) = 1$  almost everywhere. Setting $\eta_s \defeq  t(s) - W_s^{(1)}  - s$, so that, together with \eqref{tsdef}, we get \[\eta_s =  \int_{0}^s (\coth(2t(\sigma)) -1)d\sigma = \int_{0}^s \frac{2d\sigma}{e^{4t(\sigma)} -1}, \] which converges almost everywhere, as desired.
\end{proof}

Next, it will be crucial to stop the radial process before it exits the region bounded by a circle of geodesic radius $T$, and so for each $T$, we define the stopping time $\tau_T$ as follows \begin{align*}
\tau_T &\defeq \inf\{s>0 ~:~ T = d_{\bD}(0,\rho_s)\} \\ 
&~= \inf\{s>0 ~:~ T = W_s^{(1)} + s + \eta_{s}\}
\end{align*}
where the second equality follows by Lemma \ref{LFL}. Next, we will need the following lemma:
\begin{lem}\label{taublah}
For all $\omega \in X$, we have $\lim_{T\to\infty} \tau_T/T = 1$ $\bP_\omega$-almost everywhere. Moreover, we have that as $T\to\infty$, $\tau_T \to \infty$ $\bP_\omega$-almost everywhere.

\end{lem}
\begin{proof} Observe that we have $\tau_T = T - W_{\tau_T}^{(1)} - \eta_{\tau_T}$. The lemma then follows immediately from the definition of the stopping time and the law of the iterated logarithm (see \cite[Corollary II.1.12]{RY}).\end{proof}

See also \cite[Lemma 4.2]{EFL01} for related and interesting results on this stopping time.

Recall that $\bP_{\omega}$ is the Wiener measure on the space of all Brownian trajectories $W_\omega$ starting at the origin (corresponding to the random point $\omega$). Let $\bP^\theta_{\omega}$ be the Wiener measure on the space $W^\theta_\omega$ corresponding to all paths starting at the origin and conditioned to exit at the point $e^{i \theta}$ in $\partial \bD^2$. To relate the conditioned process $\rho_s^\theta$ to the unconditioned process $\rho_s$, we will need the following lemma:
\begin{lem}\label{disintegration}
\begin{align}
\bP_\omega = \frac{1}{2\pi}\int_{0}^{2\pi}\bP^{\theta}_{\omega}\, d\theta  %\\
\end{align}
\end{lem}
\begin{proof}
Recall that $W_\omega$ is the space of all hyperbolic Brownian motion trajectories starting at the origin, with $\bP_\omega$ the corresponding Wiener measure. There exists a map $\Theta : W_\omega \to \partial \bD^2$, defined $\bP_\omega$-almost everywhere, such that $\Theta(\rho) = \rho_{\infty}$, where $\rho_\infty$ is the limit point of $\rho$ on $\partial \bD^2$. It is a classical fact that the push-forward measure $\Theta_{\ast}(\bP_\omega)$ equals $\text{Leb}$, where $\text{Leb}$ is the normalized Lebesgue measure on $[0,2\pi]$. We also recall that the foliated Brownian motion and the associated harmonic measure $\hat{\nu}_*$ are in fact defined on the space $X_*=\SO(2,\bR)\backslash X$, which is almost everywhere foliated by hyperbolic disks, and so our disintegration claim follows. 
\end{proof}
\begin{rem}
See also \cite[Lemma 8]{Fran05} for a short potential theoretic proof (using Doob's $h$-process) of this fact.  We also remark that the approach to proving a central limit theorem in \cite{Fran05}, in a different context but which is modeled after the Le Jan argument in \cite{L94}, leverages the same stopping time that we employ in this paper.
\end{rem}

\begin{rem}
It is worth repeating and adapting what is written in the introduction in view of the application of the conditioned process in the sequel. The conditioned process is in fact defined on $X_* = \SO(2,\bR)\backslash X$. The space $X_*$ gives rise to a fibered space $X^{W^\theta}_*$ over $X_*$ whose fiber over each point $\omega$ in $X_*$ is $W^\theta_\omega$, and which also supports a measure $\nu_{\bP^\theta}\defeq\nu\otimes \bP^\theta$, whose conditional measure over a point $\omega$ 
is $\bP^\theta_\omega$. 

We can thus similarly define the fibered  $W^\theta$-Hodge bundle $\bP^{W^\theta}(\text{\bf{H}}^{(k)})$, whose fiber over each point $(\omega,\rho^\theta)$ in $X^{W^\theta}_*$ is~$\textbf{H}^{(k)}_\omega$. A pair $(\rho^\theta,\text{\bf{v}}) \in \bP^{W^\theta}(\textbf{H}^{(k)})$ is thus defined to be the lift of the path $\rho^\theta$ (starting at $\omega$) to $\bP^{W^\theta}(\textbf{H}^{(k)})$, obtained by parallel transport with respect to the Gauss-Manin connection. This in turn would also give rise to a measure $\hat{\nu}_{\bP^\theta}\defeq\hat{\nu}\otimes \bP^\theta$ whose conditional measure over a point $\text{\bf{v}}$ is $\bP^\theta_\omega$. 
\end{rem}

We recall the following fundamental result due to Ancona \cite{Anc90} (see also \cite[Lemma 4.1]{Gru98}):

\begin{thm}\cite[Théorème 7.3]{Anc90}\label{Ancona}
For all $\omega \in X$, and $\bP_\omega$-almost all paths $\rho$ starting at $\omega$, we have that $d_{\bD}(\rho_{0}\rho_{\infty}, \rho_{T}) = O(\log T)$ as $T\to\infty$, where $\rho_{0}\rho_{\infty}$ is the geodesic ray with $\rho_{0}\in \bD$ and $\rho_{\infty} \in \partial \bD$.
\end{thm}

Now observe that our aim is to study
\begin{align*}
\Sigma^{g}(T,[a,b]) \defeq \hat{\nu}\left(\left\{\text{\bf{v}} \in \bP(\text{\bf{H}}^{(k)}) ~:~ a \leq \frac{1}{\sqrt{T}} (\sigma_k(g_{T}, \text{\bf{v}})- T\lambda_{(k)}) \leq b\right\}\right)
\end{align*}

as $T\to\infty$.

Let 
\begin{align*}
\Sigma^{\rho}(T,[a,b]) \defeq \hat{\nu}_{\bP} \left(\left\{(\rho,\text{\bf{v}}) \in \bP^W(\text{\bf{H}}^{(k)}) ~:~ a \leq \frac{1}{\sqrt{T}} (\sigma_k(\rho_{\tau_{T}}, \text{\bf{v}})- T\lambda_{(k)}) \leq b\right\}\right)\label{trivial}
\end{align*}

\begin{lem}\label{Ancona2}
The quantity \begin{align}
|\Sigma^{g}(T,[a,b])-\Sigma^{\rho}(T,[a,b])|\to 0
\end{align} as $T\to\infty$, $\bP_\omega$-almost everywhere and for all $\omega \in X$. 
\end{lem}
\begin{proof}
Let $\text{Leb}$ denote the normalized Lebesgue (probability) measure on $[0, 2\pi]$. By applying the disintegration in Lemma \ref{disintegration}, \eqref{trivial} is also equal to 

\begin{align*}
&\Sigma^{\rho}(T,[a,b]) =   \text{Leb} \otimes  \hat{\nu}_{\bP^\theta}  \left(\left\{(\theta,\rho^\theta,\text{\bf{v}}) \in  [0,2\pi]\otimes  \bP^{W^\theta}(\text{\bf{H}}^{(k)}) ~:~ \right.\right.\\  &\left.\left. a \leq \frac{1}{\sqrt{T}} (\sigma_k(\rho^\theta_{\tau_{T}}, \text{\bf{v}})- T\lambda_{(k)}) \leq b\right\}\right) \\ &= \text{Leb} \otimes  \hat{\nu}_{\bP^\theta} \left(\left\{(\theta,\rho^\theta,\text{\bf{v}}) \in  [0,2\pi]\otimes  \bP^{W^\theta}(\text{\bf{H}}^{(k)}) ~:~ \right.\right.\\&\left.\left. a \leq \frac{1}{\sqrt{T}} \left( \sigma_k(g_T r_\theta, \text{\bf{v}})- T\lambda_{(k)} +\sigma_k(\rho^\theta_{\tau_{T}}, \text{\bf{v}})-\sigma_k(g_T r_\theta, \text{\bf{v}})\right)\leq b\right\}\right) \,.
\end{align*}

Theorem \ref{Ancona} applied to $\tau_T$ gives that, for all $\omega \in X$, $d_{\bD}(g_T r_\theta \cdot 0, \rho^\theta_{\tau_T}) = O(\log \tau_T)$ $\bP^\theta_\omega$-almost everywhere as $T\to\infty$. Together with Lemma \ref{taublah}, the lemma now follows by the Lipschitz property of the Kontsevich-Zorich cocycle (by the derivative bound in \eqref{lipschitzprop}).
\end{proof}

Therefore, it suffices to study the limiting distribution of the quantity \[\frac{1}{\sqrt{T}} (\sigma_k(\rho_{\tau_{T}}, \text{\bf{v}})- T\lambda_{(k)}).\] Observe that we have that for all $\omega \in X$, and $\bP_\omega$-almost everywhere, $\tau_T \to \infty$ as $T\to\infty$. By applying the stopping time identity $T = \tau_T + W_{\tau_T}^{(1)} + \eta_{\tau_T}$, a straightforward calculation shows the following equality:
 \begin{align}
 \frac{1}{\sqrt{T}} (\sigma_k(\rho_{\tau_T},\text{\bf{v}})- T\lambda_{(k)}) =&-\frac{1}{\sqrt{T}}\eta_{\tau_T}\lambda_{(k)} \label{etas} \\
 &-\frac{1}{\sqrt{T}}W_{\tau_T}^{(1)}\lambda_{(k)} \label{extraterm} \\
 &+\frac{1}{\sqrt{T}}(\sigma_k(\rho_{\tau_T},\text{\bf{v}}) -\tau_T \lambda_{(k)}) \label{random} \end{align}

We recall that by formulas \eqref{itointegralformula} and
\eqref{M_T_identity} we have
\begin{align*}
M_T = \int_0^T \gradient_{L_\omega} (\sigma_k(\rho_s,\text{\bf{v}})-u(\rho_s,\bfv) ) \cdot (dW^{(1)}_s,dW^{(2)}_s)
\end{align*}
and therefore
\begin{align*}
\frac{1}{\sqrt{T}} \Big(\sigma_k(\rho_{\tau_T},\text{\bf{v}}) - \tau_T\lambda_{(k)}\Big) &= \frac{1}{\sqrt{T}}\Big(u(\rho_{\tau_T},\bfv)-u(\rho_0,\bfv) +\sigma_k(\rho_0,\text{\bf{v}}) \Big) \\ &  + \frac{1}{\sqrt{T}}M_{\tau_T}\,. 
\end{align*}
Since the function $u \in L^2(\bP(\text{\bf{H}}^{(k)}),\nu)$, it follows that  
$$
\begin{aligned}
\bE_{{\hat{\nu}_\bP}} &\Big( \frac{\vert \sigma_k(\rho_{\tau_T},{\bf v} -\lambda_{(k)} \tau_T - M_{\tau_T}\vert^2}{T}\Big) \\ & \qquad= \frac{1}{T}  \bE_{{\hat{\nu}_\bP}}(\vert u(\rho_{\tau_T}, \bfv)- u(\rho_0, \bfv ) + \sigma_k(\rho_0, \bfv)\vert^2)
\end{aligned}
$$
converges to $0$. Thus the random variable $(\sigma(\rho_T,{\bf v}) -\lambda_{(k)} T-M_{\tau_T})/\sqrt{T}$ converges to $0$ in square mean, hence in distribution.

\smallskip
It suffices then to study the asymptotic distribution of \begin{align*}
 -\frac{1}{\sqrt{T}}&W_{\tau_T}^{(1)}\lambda_{(k)}+\frac{1}{\sqrt{T}}M_{\tau_T} \\ &= \frac{1}{\sqrt{T}}\int_0^{\tau_T}(-\lambda_{(k)}+\frac{\partial\sigma_k}{\partial t}(\rho_s, \bfv)- \frac{\partial u}{\partial t}(\rho_s,\bfv))dW_s^{(1)}\\
 &+\frac{1}{\sqrt{T}}\int_0^{\tau_T}\frac{2}{\sinh(2t(s))}(\frac{\partial\sigma_k}{\partial \theta}(\rho_s, \bfv)- \frac{\partial u}{\partial \theta} \,.(\rho_s,\bfv))dW_s^{(2)}\end{align*} 

It follows from Theorem \ref{flj:mclt} that the
law of the random variable $$ -\frac{1}{\sqrt{T}}W_{\tau_T}^{(1)}\lambda_{(k)}+\frac{1}{\sqrt{T}}M_{\tau_T}$$  converges, as \(t \to \infty\), towards the centered Gaussian law with variance  
\[
V^{(k)}_{g_\infty}= \int_{\bP(\text{\bf{H}}^{(k)})} \Big( \left\vert -\lambda_{(k)}+\frac{\partial\sigma_k}{\partial t}- \frac{\partial u}{\partial t}  \right\vert^2 + \left\vert\frac{2}{\sinh(2t)}(\frac{\partial\sigma_k}{\partial \theta}- \frac{\partial u}{\partial \theta}) \right\vert^2  \Big)\, d\hat\nu\,.
\]
To conclude the proof of  Theorem \ref{dclt}, we simplify the
expression of the variance. In fact,
$$
V^{(k)}_{g_\infty}= \lambda_{(k)}^2 -2 \lambda_{(k)}\int_{\bP(\text{\bf{H}}^{(k)})} \left(\frac{\partial\sigma_k}{\partial t}-\frac{\partial u}{\partial t}\right) d \hat \nu + \int_{\bP(\text{\bf{H}}^{(k)})} | \gradient_{L_\omega}\sigma_k - \gradient_{L_\omega}u|^2 d \hat \nu \,.
$$
We recall that by formula \eqref{variance} we have
$$
V^{(k)}_{\rho_\infty}=\int_{\bP(\text{\bf{H}}^{(k)})} | \gradient_{L_\omega}\sigma_k - \gradient_{L_\omega}u|^2 d \hat \nu =\int_{\bP(\bfH^{(k)})} |\Psi_k -\gradient_{L_\omega}u|^2 d\hat{\nu} \,. 
$$
In addition, since $u \in L^2(\bP(\text{\bf{H}}^{(k)}), d\hat\nu)
\subset L^1(\bP(\text{\bf{H}}^{(k)}), d\hat\nu)$ and the measure
$\hat \nu$ is invariant under the Teichm\"uller geodesic flow, we have
$$
\int_{\bP(\text{\bf{H}}^{(k)})} \frac{\partial u}{\partial t} \, d \hat \nu =0\,.
$$
since the radial derivative in each Teichm\"uller disk coincides with
the Lie derivative of the Teichm\"uller flow.

Finally, since the measure $\hat \nu$ is supported on the unstable Oseledets subbundle $E^+_k$ over $X$, we have
$$
\int_{\bP(\text{\bf{H}}^{(k)})} \frac{\partial\sigma_k}{\partial t} d \hat \nu =\lambda_{(k)}\,.
$$
The above identity can be proved as follows. By the ergodic theorem, and by Oseledets theorem, for $\nu$-almost all $\omega \in X$, we have
$$
\begin{aligned}
\int_{\bP(\text{\bf{H}}^{(k)})} \frac{\partial\sigma_k}{\partial t} d \hat \nu &= \lim_{T\to \infty}\frac{1}{T} \int_0^T \frac{\partial\sigma_k}{\partial t} (g_t(\omega, E^+_k(\omega)) dt \\ & = \lim_{T\to \infty} \frac{1}{T} \Big(\sigma_k(g_T(\omega, E^+_k(\omega)) -
\sigma_k(\omega, E^+_k(\omega)) \Big) = \lambda_{(k)}\,.
\end{aligned} $$
In conclusion we have proved that indeed
\begin{align} V^{(k)}_{g_\infty}= V^{(k)}_{\rho_\infty} + - 2\lambda_{(k)}^2 + \lambda_{(k)}^2 = V^{(k)}_{\rho_\infty}- \lambda_{(k)}^2 \,.\label{dvarianceformula}
\end{align} \hfill\qedsymbol

\section{Positivity of the variance}\label{posvar}
\subsection{Random cocycle}\label{vrclt}
Recall that \eqref{dvarianceformula} says that $V^{(k)}_{g_\infty} = V^{(k)}_{\rho_\infty}  -\lambda_{(k)}^2$, and so we also have the following important corollary: 
\begin{cor}
If $\lambda_k>\lambda_{k+1}$, then $V^{(k)}_{\rho_\infty} > 0$. \end{cor}
\begin{proof}
Since, by construction, $V^{(k)}_{g_\infty}\geq 0$, and we have that $V^{(k)}_{\rho_\infty} \geq \lambda_{(k)}^2  >0$, and it is clear that, since $\lambda_{(k)} = \sum_{i=1}^k\lambda_i$, we have $\lambda_{(k)}^2 \geq  \lambda_1^2>0$.
\end{proof}

\subsection{Deterministic cocycle}
\label{posvar:det}
While \eqref{dvarianceformula} ensures convergence of the asymptotic variance for the deterministic cocycle, it is not clear to us how it can be leveraged to deduce its positivity. Instead, we approach the positivity of the variance for the deterministic cocycle directly, in the spirit of the potential theoretic approach in \cite{For02}. We first observe that a direct expression of the converging asymptotic variance for the deterministic cocycle is 
\begin{align}
\label{eq:det_var}
V^{(k)}_{g_\infty} = \lim_{T\to\infty} \frac{1}{T} \int_{\bP(\textbf{H})} \left[\sigma_k(g_T, \text{\textbf{v}}) -\lambda_{(k)} T \right]^2 \,d\hat{\nu}
\end{align}

The existence and the regularity of the solution 
$u$ of the Poisson equation \eqref{poissoneq} will be again crucial  for our approach towards the positivity of $V^{(k)}_{g_\infty}$. 

\smallskip
\noindent Let $F_k(g_T r_\theta,\text{\textbf{v}}) \defeq \sigma_k(g_T r_\theta, \text{\textbf{v}}) -\lambda_{(k)} T$. In fact, we will study an auxiliary random variable $F_k- u$, and use it at the end to deduce the positivity of the asymptotic variance $V^{(k)}_{g_\infty}$. 

\smallskip 
\noindent Let $\Psi_k$ the vector valued function defined in formula \eqref{eq:PhiPsi}: 
$$
\Psi_k(\omega, {\bf v}) = \text{tr} \left(B^{(k)}_\omega ({\bf v}) \right)\,,  \qquad \text{ for all } (\omega, {\bf v}) \in \bP(\textbf{H}^{(k)})   \,,.
$$
We prove below the following condition for the vanishing of the deterministic variance.
\begin{lem}
\label{lem:det_var}
The variance $V^{(k)}_{g_\infty}$  of the deterministic cocycle (see formula \eqref{eq:det_var}) vanishes if and only if
$$
\Psi_k \circ E^+_k - (\lambda_{(k)},0) -\gradient_{L_\omega} u \circ E_k^+ =0\,  \qquad \text{$\nu$-almost everywhere}\,.
$$
\end{lem}
\begin{proof}
We first prove that the normalized asymptotic variance of the random variable $F_k$ coincides with that of $F_k -u$.

It follows by an immediate application of \cite[Lemma 3.1]{For02} that, for any smooth function $F$ and any function $u\in W^{2,\infty}$ on the Poincar\'e
disk, with respect to hyperbolic geodesic polar coordinates $z=(t,\theta)$, we have the formula

\begin{align*}
\frac{1}{2\pi} \frac{\partial}{\partial t} \int_0^{2\pi} (F-u)^2(t,\theta)d\theta &= \frac{1}{2} \tanh(t) \frac{1}{|D_t|} \int_{D_t} \laplacian_{L_\omega}((F-u)^2) \omega_P \\
 &= \tanh(t) \frac{1}{|D_t|} \int_{D_t} (F-u)\laplacian_{L_\omega}(F-u) \omega_P \\&+ \tanh(t) \frac{1}{|D_t|} \int_{D_t} |\gradient_{L_\omega}(F-u)|^2 \omega_P
\end{align*}
where $|D_t|$ is the hyperbolic area element of the disk $D_t$ of geodesic radius $t>0$ that is centered at the origin, and $\omega_P$ is the hyperbolic area on the Poincaré disk. 

\smallskip
By applying the above formula to the function $F(t,\theta)= F_k (g_t r_\theta, {\bf v})$, for every ${\bf v} \in \bP(\mathbf{H}^{(k)})$ (we recall that $\bP(\mathbf{H}^{(k)})$ denotes the projective Hodge bundle over an $\SL(2, \R)$-invariant sub-orbifold $X$ of the moduli space of Abelian differentials)
and $U(t, \theta, \bfv)  = u(g_t r_\theta (\omega), \bfv)$  and by integrating over $\bP(\mathbf{H}^{(k)})$ with respect to the  harmonic measure $\hat \nu_*$, we have
\begin{align*}
\int_{\bP(\mathbf{H^{(k)}})}& \frac{\partial}{\partial t} (F_k-U)^2 (t,\theta, \bfv)d\hat{\nu}_* \\ &= \int_{\bP(\mathbf{H}^{(k)})} \frac{\tanh(t)}{\sinh^2(t)} \int_{0}^{t} (F_k-u)\laplacian_{L_\omega}(F_k-U)(\tau,\theta, \bfv ) d(\sinh^2\tau)d\hat{\nu}_*  \\&+ \int_{\bP(\mathbf{H}^{(k)})} \frac{\tanh(t)}{\sinh^2(t)} \int_{0}^{t} |\gradient_{L_\omega}(F_k-U)|^2(\tau,\theta, \bfv) d(\sinh^2\tau)d\hat{\nu}_*
\end{align*}

By integrating over $[0,T]$ with respect to $dt$, we have

\begin{align}
\frac{1}{T}&\left[\int_{\bP(\mathbf{H}^{(k)})}[(F_k-u)^2(g_T, \textbf{v})-(F_k-u)^2(g_0, \textbf{v})]d\hat{\nu}_*\right] \\&=  \frac{1}{T} \int_0^T \int_{\bP(\mathbf{H}^{(k)})} \frac{\tanh(t)}{\sinh^2(t)} \int_{0}^{t} (F_k-U)\laplacian_{L_\omega}(F_k-U) d(\sinh^2\tau)d\hat{\nu}_*   dt\label{zeroterm} \\&+ \frac{1}{T} \int_0^T \int_{\bP(\mathbf{H}^{(k)})} \frac{\tanh(t)}{\sinh^2(t)} \int_{0}^{t} |\gradient_{L_\omega}(F_k-U)|^2 d(\sinh^2\tau)d\hat{\nu}_* dt
\end{align}
By Eq. \eqref{poissoneq}, we observe that 
\begin{align*}
&\laplacian_{L_\omega}(F_k-U)(t,\theta, \bfv) \\ 
& \quad = \laplacian_{L_\omega}(\sigma_k(g_t r_\theta\omega, \text{\textbf{v}})) -\laplacian_{L_\omega}(\lambda_{(k)}  t)
\\& \quad =  2\lambda_{(k)} (1- \coth(2t))  \to 0 \end{align*}
exponentially as $t\to\infty$. It follows therefore that \eqref{zeroterm} converges to 0 as $t\to \infty$, and we thus have that 
\begin{align*}
\lim_{T\to\infty}\frac{1}{T}\int_{\bP(\mathbf{H}^{(k)})}&(F_k-u)^2(g_T, \textbf{v})d\hat{\nu}_*\\ &=\int_{\bP(\mathbf{H}^{(k)})} |\Psi_k(E^+_k(\omega))- (\lambda_{(k)},0) -\gradient_{L_\omega} u(\omega,E^+_k(\omega))|^2 d\nu \,.
\end{align*}
\begin{rem}
The above steps follow closely the outline of the proof of \cite[Theorem 1]{FMZ12}, which we refer to for more details.
\end{rem}

We have therefore shown that the normalized asymptotic variance of $F_k-u$ is strictly positive if the function $ \Psi_k \circ E^+_k - \lambda_{(k)} -\cD u$ is not identically zero. The final claim in the argument is that the asymptotic variance of $F_k$ is no smaller than that of $F_k-u$, and this follows by an immediate application of the triangle inequality, as follows
\begin{align*}
\left[\frac{1}{T}\int_{\bP(\mathbf{H}^{(k)})} (F_k-u)^2(g_T, \textbf{v})d\hat\nu_* \right]^{1/2}
&\leq  \left[\frac{1}{T}\int_{\bP(\mathbf{H}^{(k)})} F_k^2(g_T, \textbf{v})d\hat \nu_* \right]^{1/2} \\ &+  \left[\frac{1}{T}\int_{X} u^2(g_T\omega, \bfv)d\hat\nu_* \right]^{1/2} \\
&=  \left[\frac{1}{T}\int_{\bP(\mathbf{H}^{(k)})} F_k^2(g_T, \textbf{v})d\hat \nu_* \right]^{1/2} \\ &+  \frac{1}{\sqrt{T}}\|u\|_{L^2(\bP(\bfH^{(k)}),\hat \nu)} 
\end{align*}
together with the square integrability of $u$. We have therefore shown that if the asymptotic variance for the deterministic cocycle $V^{(k)}_{g_\infty}$ is  equal to
zero, then 
\begin{align}
\Psi_k \circ E^+_k(\omega) - (\lambda_{(k)},0) -\gradient_{L_\omega}u\circ E^+_k(\omega) =0  \quad \nu\text{-almost everywhere.} \label{zerovariancecond}
\end{align}
\end{proof}

Now let $\{X,Y, \Theta\}$ be the standard generators of the Lie algebra of $\SL(2, \R)$ corresponding to the geodesic flow, the orthogonal geodesic flow and the maximal compact subgroup
 $\SO(2, \R)$, given by the formulas: \begin{align*}X=\begin{pmatrix} 1/2&0\\ 0&-1/2 \end{pmatrix}, \quad Y=\begin{pmatrix} 0&1/2\\ 1/2&0 \end{pmatrix}, \quad \Theta=\begin{pmatrix} 0&-1/2\\ 1/2&0 \end{pmatrix}.\end{align*}

In order to prove a crucial regularity result for the unstable Oseledets subspace under the assumption of zero variance, it will be useful to derive the following lemma for the lift of our $\SO(2,\bR)$-invariant function $u$ from $\SO(2,\bR)\backslash \bP(\bfH)$ to $\bP(\bfH)$ (which we continue to call $u$ by abuse of notation): 

\begin{lemma} The hyperbolic gradient (in the radial and tangential directions) and Laplacian is given (in the weak sense) by the following formulas:
$$
\begin{aligned}
\nabla_{L_\omega}  u(g_t r_\theta, \bfv) &= 2 \left (X u, Y u \right) (g_t r_\theta,  \bfv) \, \\ \laplacian_{L_\omega} u(g_t r_\theta, \bfv)  &= 4 (X^2 + Y^2)u (g_t r_\theta, \bfv) \, . 
\end{aligned}
$$
\end{lemma}
\begin{proof}
By definition we have 
$$
\nabla_{L_\omega}  u(g_t r_\theta, \bfv) = 2(Xu) (g_t r_\theta \omega, \bfv)\,.
$$
The computation of the angular derivative
is based on the formula
$$
\begin{aligned}
 \exp (\theta \Theta)\exp (2t X)  &=
\exp (2 tX) \exp (-2tX) \exp (\theta \Theta) \exp(2tX)  \\
&=
\exp (2tX) \text{Ad}_{\exp (-2 tX)}( \exp (\theta \Theta))  \\
&=\exp (2tX) \exp ( e^{\text{ad}_{-2tX}}  (\theta\Theta))  \\ &=  \exp (2tX) \exp(\theta ( \cosh (2t) \Theta + \sinh(2t) Y    ) ) \,.
\end{aligned}
$$
The above formula is computed with respect to standard generators $\{X,Y,\Theta\}$ of
the Lie algebra $\mathfrak{sl}(2,\R)$ which
satisfy the commutation relations
$$
[\Theta, X]=Y \,, \quad [\Theta, Y]=-X \,, \quad [X,Y] = -\Theta \,.
$$
Under the convention in \cite{For02}, the curvature of the Poincar\'e plane is taken to be $-4$, which corresponds to the choice of generators $\{2X, 2Y, \Theta\}$. 

\smallskip
It follows that 
$$
\frac{\partial}{\partial \theta}  u(g_t r_\theta, \bfv) =   \sinh(2t) Y  u(g_t r_\theta, \bfv)\,.
$$
We can now compute the hyperbolic gradient and Laplacian. We have
$$
\begin{aligned}
\nabla_{L_\omega} u(g_t r_\theta, \bfv) &=
\left(\frac{\partial}{\partial t }  u(g_t r_\theta, \bfv),
\frac{2}{\sinh(2t)} \frac{\partial}{\partial \theta}  u(g_t r_\theta, \bfv)\right)
 \\&= 2 (Xu, Y u) (g_rr_\theta \omega,\bfv)\,.
\end{aligned}
$$
By the commutation relation $[\Theta, Y]=-X$, we also have
$$
\begin{aligned}
&\laplacian_{L_\omega} u(g_t r_\theta, \bfv) =\left(
\frac{\partial^2}{\partial t^2} + 2\coth(2t)\frac{\partial}{\partial t}+\frac{4}{\sinh^2(2t)}\frac{\partial^2}{\partial \theta^2}\right) u(g_t r_\theta, \bfv) \\ 
& \quad= (4 X^2 + 4\coth(2t)X + \frac{4}{\sinh^2(2t)} (\cosh(2t) \Theta + \sinh(2t) Y)^2)u(g_t r_\theta, \bfv)
\\
& \quad= (4 (X^2 + Y^2 + \coth^2(2t)\Theta^2 + 2 \coth (2t) Y \Theta))u(g_t r_\theta, \bfv)
\\
& \quad= 4 (X^2 + Y^2)u(g_t r_\theta,\bfv) \,.
\end{aligned}
$$

The computation is completed.
\end{proof}

\smallskip
\noindent Let $h^-_t = \begin{pmatrix} 1 & 0 \\ t & 1 \end{pmatrix}$ denote the stable horocyclic (unipotent) subgroup of the group $\SL(2,\R)$.
In our notation, the generator of the flow $h^-_t$ is the vector field $H:=Y+ \Theta$.
In fact, we have
$$
[X, Y+\Theta] =- \Theta -Y = -(Y+\Theta)\,
$$
and we follow the convention that $\SL(2,\R)$ acts on itself and its quotients on the right by multiplication on the left.

\begin{lemma} 
\label{lemma:Lipschitz} Assume the variance $V^{(k)}_{g_\infty} =0$.  Then for $\nu$-almost all $\omega \in X$ the function
$$
\Psi_k (E^+_k(h^-_s \omega) ) \, , \quad s \in \R\,,
$$
is a Lipschitz function. 
\end{lemma}
\begin{proof}
Since by assumption the variance $V^{(k)}_{g_\infty} =0$ the identity \eqref{zerovariancecond} holds, hence in particular we have
$$
X  u \circ E^+_k(\omega) = \frac{1}{2} \left( \Re \,(\Psi_k \circ E^+_k) -\lambda_{(k)} \right)\,.
$$
Since the unstable space $E^+_k$ of the cocycle is invariant under the Teichm\"uller flow, it follows that the function $\Psi_k \circ E^+_k$ is
for almost all $\omega \in X$ smooth along the Teichm\"uller orbit, and 
by a similar argument along the orbit of the unstable Teichm\"uller horocycle flow. It follows then that the function $u$ is infinitely differentiable, with derivatives uniformly bounded almost everywhere, along 
the geodesic flow orbit for almost all $\omega\in X$. 

By the construction of the function $u$ as a solution of the equation 
$$
 \left(X^2 + Y^2 \right) u \circ E^+_k(\omega) = \frac{1}{4}\laplacian_{L_\omega} u \circ E^+_k(\omega) = \frac{1}{2} (\Phi_k\circ E^+_\omega -\lambda_{(k)}) \,,
$$

since $Xu\circ E^+_k$, $X^2 u \circ E^+_k$ and $\laplacian_{L_\omega} u \circ E^+_k$ are bounded and, by definition $H= Y+\Theta$, it follows that, for $\nu$-almost all
$\omega \in X$, we have
\begin{align*}
\frac{d^2}{ds^2} u \circ E^+_k (h_s^{-} \omega)= H^2  u \circ E^+_k (h^-_s\omega) &= YH  u \circ E^+_k (h^-_s\omega) \\ &=Y^2  u \circ E^+_k (h^-_s\omega)
\end{align*}
is a bounded function, which in turn implies
that
$$
 \frac{1}{2}\,\Im (\Psi_k \circ E^+_k) (h^-_s \omega) = Hu \circ E^+_k(h^-_s \omega) 
$$

has a uniformly bounded derivative, hence it is a Lipschitz function.  For the 
real part of the function, we argue that
$$
\begin{aligned}
H \Re( \Psi_k \circ E^+_k) &= H  X u \circ E^+_k  = [H,X] u \circ E^+_k + XH u \circ E^+_k \\ & = H u\circ E^+_k + X Hu\circ E^+_k  =
 \frac{1}{2} (I+ X)\Im( \Psi_k \circ E^+_k)
\end{aligned}
$$
which is again a function uniformly bounded almost everywhere, hence the function
$\Re( \Psi_k \circ E^+_k)$ is also Lipschitz
along almost all horocycle orbits.
\end{proof}

Let ${\bf H}$ be an
$\SL(2, \R)$-invariant, symplectic subbundle of the Hodge bundle, symplectically orthogonal to the tautological subbundle.

\begin{cor}\label{cor:freezing}
  If  the Lyapunov spectrum of the Kontsevich--Zorich cocycle on
$\bf H$ is simple, that is, if $\lambda_1> \dots >\lambda_h$,
then the deterministic Kontsevich--Zorich cocycle on ${\bf H}^{(1)}= {\bf H}$ has strictly positive variance, that is,  $V^{(1)}_{g_\infty} > 0$. 
\end{cor}

\begin{proof} The proof is by contradiction. 

\smallskip
\noindent Let us assume that $V^{(1)}_{g_\infty}=0$. 
By Lemma~\ref{lemma:Lipschitz} for $\nu$-almost all $\omega \in X$ the function
$$
\Psi_1 (E^+_1(h^-_s \omega) ) \, , \quad s \in \R\,,
$$
is a Lipschitz function. 

The strategy of the argument, based on the so-called {\it freezing argument} from \cite{CF20}, consists in deriving from the above Lipschitz property the existence of a proper $\SL(2,\bR)$-invariant subbundle of $\bP({\bf H})$, thereby contradicting the strong irreducibility assumption.

\smallskip
\noindent Recall that the function $\Psi_1$ is given by
the formula (see formula \eqref{eq:PhiPsi})
$$
\Psi_1 (\omega, \text{\bf{v}})=  \text{tr}(B^{(1)}_\omega(\text{\bf{v}})) = B_\omega(\text{\bf{v}}) \,,
\quad \text{ for } (\omega, {\bf v}) \in {\bf H}\,.
$$
Since for every $\omega \in X$ the matrix $B_\omega$ is a complex symmetric matrix, with
entries given by a complex quadratic form, 
it follows that the function $\Psi_1$ is a quadratic polynomial function (with respect to projective coordinates) on every fiber ${\bf H}_\omega$.
In addition, the function $\Psi_1$ is non-constant along circle orbits, since
$$
\Psi_1 (r_{\theta}\omega, \text{\bf{v}}) =
e^{-2i \theta} \Psi_1 (\omega, \text{\bf{v}}) \,, \quad \text{ for all } {\bf v} \in {\bf H}_\omega\,.
$$
We define a measurable subbundle of the bundle 
${\bf H}$ as follows. Since by assumption the Lyapunov exponent
$\lambda_1 >0$, it follows that
$$
\vert \Psi_1(E^+_1(\omega)) \vert >0 \quad \text{almost everywhere.}
$$
In fact, otherwise the second identity in formula \eqref{laplacianformula} (for $k=1$), since the bundle $E^+_1$ is invariant under the Teichm\"uller flow, would imply that $\lambda_1=0$. It then follows that there exists $c>0$ and a compact set $\mathcal K \subset X$ such that
\begin{equation}
\label{eq:norm_lbound}
\min_{\omega \in \mathcal{K}} \vert \Psi_1(E^+_1(\omega)) \vert \geq c>0\,.
\end{equation}

 For every $\omega\in X$
Birkhoff regular for the Teichm\"uller geodesic flow and Oseledets regular for the Kontsevich--Zorich cocycle, and for every forward return time $t > 0$ of the Teichm\"uller geodesic flow to the compact set $\cK \subset X$, we let 
\begin{equation}
\label{eq:cW}
\cW (g_{t}\omega):=  \Psi_1^{-1} \left\{ \Psi_1 (E^+_1 (g_{t} \omega)) \right\}\,.
\end{equation}
We note that $\cW(g_{t}\omega)$ is a real analytic submanifold of (real) codimension $2$
which contains the point $E^+_1(g_{t}\omega) \subset  \bP (H_{g_{t}\omega})$.
We then let, for all $\omega \in X$,
$$
\cV (\omega) = \bigcap_{s\geq 0}  \overline{\bigcup_{t\geq s}  \{  g_{-t} \left( \cW(g_{t}\omega)\right) \vert g_{t}\omega \in \cK\} } .
$$
We remark that since $\lambda_1>0$, by Oseledets theorem the set $\cV(\omega)$ is contained in a finite union of $g_t$-invariant subspaces.

By definition, for $\nu$-almost all $\omega \in X$, the set $\cV(\omega) 
\subset \bP({\bf H}_\omega)$ is  closed and non-empty, since it contains $E^+_1(\omega)$. It is straightforward to derive from its definition that $\cV$ is a (measurable) $\{g_t\}$-invariant subset. 

The crucial point of the argument is to prove that
$\cV$ is invariant under the stable Teichm\"uller  horocycle flow $\{h^-_s\}$.

Let ${\bf v} \in g_{-t} \left( \cW(g_{t}\omega)\right)$. By definition,
$\Psi_1 (g_{t} ({\bf v})) \in \Psi_1 (E^+_1(g_{t}(\omega)))$. There exists a constant $C_\cK$ such that for any fixed $r>0$
the distance 
$$
d(g_{t}\omega, g_{t} (h^-_r \omega))= 
d(g_{t}\omega, h^-_{e^{-2t} r} (g_{t} \omega)) \leq C_\cK r  e^{-2t} \,,
$$
hence by the Lipschitz property of the function
$\Psi_1 \circ E^+_1$ (which holds by Lemma 
\ref{lemma:Lipschitz}) we have that there exists
a constant $C'_\cK$ such that
$$
\Vert (\Psi_1 \circ E^+_1)(g_{t}\omega) -  (\Psi_1 \circ E^+_1) (g_{t} (h^-_r \omega)))\Vert  \leq C'_\cK r e^{-2t}\,.
$$
We also have (with respect to the Hodge metric)
$$
d \left( g_{t} (h^-_r ({\bf v})), g_{t} ({\bf v})\right) \leq r e^{-2t}\,,
$$
so that we have the estimate
$$
\begin{aligned}
&\Vert (\Psi_1(g_{t} (h^-_r {\bf v})) - 
\Psi_1(E^+_1 ( g_{t} h^-_r \omega)) \Vert \\ & \quad \leq 
\Vert \Psi_1(g_{t} (h^-_r {\bf v})) - 
\Psi_1(g_{t} ({\bf v})) \Vert  + 
\Vert  \Psi_1(g_{t} ({\bf v})) - \Psi_1(E^+_1 ( g_{t} h^-_r \omega)) \Vert  \\ & \quad =
\Vert \Psi_1(g_{t} (h^-_r {\bf v})) - 
\Psi_1(g_{t} ({\bf v})) \Vert  + 
\Vert  \Psi_1( E^+_1 (g_{t} (\omega)) - \Psi_1(E^+_1 ( g_{t} h^-_r \omega)) \Vert \\
& \quad \leq \Vert D\Psi_1 \Vert_\cK\, d\Big(g_{t} (h^-_r {\bf v}), g_{t} ({\bf v})\Big) +
\Vert  \Psi_1( E^+_1 (g_{t} (\omega)) - \Psi_1(E^+_1 ( g_{t} h^-_r \omega)) \Vert\,.
\end{aligned}
$$
hence there exists a constant $C''_\cK>0$ so that we have the inequality
$$
\begin{aligned}
&\Vert (\Psi_1(g_{t} (h^-_r {\bf v})) - 
\Psi_1(E^+_1 ( g_{t} h^-_r \omega)) \Vert \leq C^{''}_\cK r e^{-2t}\,.
\end{aligned}
$$

It follows that, by the lower bound in formula \eqref{eq:norm_lbound} and by Lemma \ref{lemma:crit}, for sufficiently large $t>0$, there exists a constant $C^{(3)}_\cK>0$ such that
$$
d \left( g_{t} (h^-_r {\bf v}), \cW (g_{t} h^-_r \omega)\right)  \leq C^{(3)}_\cK r e^{-2t}\,.
$$
Next we claim that there exists $\lambda <1$ such that the above estimate implies that there exists
a constant $C^{(4)}_\cK >0$ such that
$$
d\Big(h^-_r {\bf v}, g_{-t}\cW (g_{t} h^-_r \omega)\Big)  \leq C^{(4)}_\cK r e^{-2(1-\lambda)t}\,.
$$
The above conclusion follows from the fact that there exists $\lambda \in (0,1)$ such that,
on the symplectic orthogonal of the tautological bundle, the Lyapunov spectrum is contained in the interval $(-\lambda, \lambda)$. It then follows by Oseledets theorem that for every Birkhoff generic and regular $\omega\in X$ and for every ${\bf v}$, ${\bf w} \in \bP(\bfH_\omega)$,
$$
\limsup_{t\to +\infty} \frac{1}{t} \log d(g_t({\bf v}), g_t({\bf w}))  \leq 2 \lambda\,.
$$
We have thus proved that, for every $s>0$,
$$
h^-_r ({\bf v}) \in \overline{\bigcup_{t\geq s}  \{  g_{-t} \left( \cW (g_{t}h^-_r\omega)\right) \vert g_{t}h^-_r\omega \in \cK\} }\,,
$$
hence $h^-_r ({\bf v}) \in \cV (h^-_r\omega)$, for all  ${\bf v} \in g_{-t} (\cW(g_{t} \omega)))$.  Since $\cV ( h^-_r\omega)$ is closed, it follows that, for every $r\in \R$,
$$
h^-_r \left( \cV (\omega) \right)
\subset \cV (h^-_r\omega)\,,
$$
and, since the reverse inclusion can be proved by reversing the time in the horocycle flow, we have proved the invariance of the bundle $\cV$
under the unstable Teichm\"uller horocycle flow.

\smallskip
\noindent  We claim that by the construction of the bundle $\cV$, the unstable bundle $E^+_1 \subset \cV$. In fact, by the definition in formula \eqref{eq:cW} we have that, for almost all $\omega \in X$ and for all $t\in \R$, 
$$
E^+_1 (g_{t} \omega) \subset \cW(g_{t}\omega)\,.
$$
and since $E^+_1$ is a $\{g_t\}$-invariant bundle, it follows that, for all $t\geq 0$,
$$
E^+_1 (\omega) \subset g_{-t} \cW(g_{t}\omega)
$$
which implies the claim.

 We can then define an $\SL(2, \R)$-invariant subbundle as follows.  Let $\cE$ denote the smallest measurable  forward $\{h^-_r\}$-invariant bundle  which contains $E^+_1$.  In other terms, for almost
all $\omega \in X$, let 
$$
\cE(\omega) := \sum_{r\geq 0} 
h^-_r E^+_1 (h^-_{-r} \omega) \,.
$$
We note that, by the above definition, $\cE \subset \cV$ since $E^+_1 \subset \cV$, and the latter bundle is  $\{h^-_r\}$-invariant (as well as $\{g_t\}$-invariant).

We then prove that the bundle $\cE$
is $\SL(2,\bR)$-invariant. It is clearly forward $\{h^-_r\}$-invariant by definition. Let us
prove that it is $\{g_t\}$-invariant. 

By the commutation relation and by the $\{g_t\}$-invariance of the bundle $E^+_1$,
for almost all $\omega \in X$ and for all $t, r\in \R$, we have
$$
\begin{aligned}
g_t \left( h^-_r E^+_1 (h^-_{-r} \omega)\right) &=
(h^-_{e^{-t} r}\circ g_t ) E^+_1 (h^-_{-r} \omega) \\ &=  h^-_{e^{-t} r} \left( E^+_1 ( g_t\circ  h^-_{-r} \omega) \right) =
h^-_{e^{-t} r} E^+_1 ( h^-_{-e^{-t}r} \circ g_t\omega)
\end{aligned}
$$
which immediately implies that, for all $t\in \R$,
$$
g_t \cE(\omega) = \cE (g_t\omega)\,,
$$
hence the bundle $\cE$ is (forward and backward) $\{g_t\}$-invariant.

\noindent Let us then prove that the bundle $\cE$ is forward $\{h^+_s\}$-invariant. For the unstable horocycle flow $\{h^+_s\}$ we have the following commutation relations. For every $r, s \in \R$, with $rs\not=-1$, let
$$
\rho(r,s) = \frac{r}{1+rs} \,, \quad \sigma(r,s) = s(1+rs) \,, \quad \tau(r,s)= \log (1+rs)\,. 
$$
We then have the commutation relations:
$$
h^+_s \circ h^-_{r} = h^-_\rho \circ h^+_\sigma \circ g_\tau \,.
$$
Since $E^+_1$ is $\{g_t\}$-invariant and $\{h^+_s\}$-invariant, it follows that we have
$$
\begin{aligned}
 h^+_s \left( h^-_r E^+_1( h^-_{-r} \omega)\right) & = h^-_\rho \circ h^+_\sigma \circ g_\tau  \left(E^+_1( h^-_{-r} \omega)\right) \\ & =
 h^-_\rho    \left(E^+_1( h^+_\sigma \circ g_\tau \circ h^-_{-r} \omega)\right)
  = h^-_\rho    \left(E^+_1( h^-_{-\rho} ( h^+_s  \omega) )\right)\,,
\end{aligned}
$$
which immediately implies that $\cE$ is
forward $\{h^+_s\}$-invariant. Finally, since
$\cE$ is $\{g_t\}$-invariant and forward
$\{h^\pm_s\}$-invariant, it follows that it is
$SL(2, \R)$-invariant as claimed.

\smallskip
\noindent Finally we remark that by the condition that  the Lyapunov spectrum is simple, it follows that  $\cV \not =\bP(\bf H)$, hence $\cE \subset \cV \not = \bP(\bf H)$. In fact, by definition, since for almost all $\omega \in X$ and for all $t \geq 0$, the real analytic sets $\cW (g_t \omega)$ have positive codimension equal to $2$, by Oseledets theorem, for almost all $\omega \in X$, the subset $\cV$ is contained in the union of finitely many proper $g_t$-invariant sub-bundles of $\bP ({\bf H})$, given by the Oseledets decomposition, namely all the codimension $2$ sums of the one-dimensional  Oseledets subspaces, in contradiction with the hypothesis that $\bf H$ is strongly irreducible.  
\end{proof}

\section{A Central Limit Theorem for Generic Sections}\label{sectionsclt}
Since our main results randomizes \emph{both} the Abelian differential $\omega\in X$ and vector $\bfv_\omega\in 
\bP(\bfH_\omega^{(k)})$ with respect to the measure $\nu^*$, it is natural to ask if our results also hold for (suitably defined) sections of $\bP(\bfH^{(k)})$. Following \cite{AF24}, we introduce the following

\begin{definition}\label{fut-gen}
We say that $\bfv=(\omega,\bfv_{\omega})$ is \emph{future-Oseledets-generic} for $\nu$ a.e. $\omega\in X$ if $\bfv$ is a (measurable) section $\bfv~:~X \to \bP(\bfH^{(k)})$ of $\bP(\bfH^{(k)})$ such that, for $\nu$-a.e. $\omega \in X$,
$$\lim_{T\to\infty} \frac{1}{T} \sigma_k(g_{T}, \bfv_\omega) = \sum_{i=1}^k \lambda_i \,.$$ . 
\end{definition}

In particular, for the deterministic cocycle, it is straightforward to derive

\begin{cor}
Under the hypothesis of Theorem \ref{dclt}, there exists a real number $V^{(k)}_{g_\infty} \geq 0$ such that for any future-Oseledets-generic section $\bfv=(\omega,\bfv_{\omega})$ of $\bP(\bfH^{(k)})$, we have
\begin{align*}
\lim_{T\to\infty} \nu\left(\left\{\omega\in X ~:~ a \leq \frac{1}{\sqrt{T}} \left(\sigma_k(g_{T}, \bfv_\omega)- T(\sum_{i=1}^k \lambda_i) \right) \leq b\right\}\right) 
\\= \frac{1}{\sqrt{2\pi V^{(k)}_{g_\infty}}}\int_a^b\exp(-x^2/ V^{(k)}_{g_\infty})dx.
\end{align*}
Moreover, if the Lyapunov spectrum is simple, then $V^{(1)}_{g_\infty} > 0$.
\end{cor}

\begin{proof}
By Oseledets' theorem, for $\nu$-a.e. $\omega$, for all $\bfv \in \bP(\text{\bf{H}}^{(k)})$, there exist constants $C>0$ and $\lambda >0$ such that, for all $t>0$,
$$
\text{dist}_{g_t\omega}  (\bfv_{g_t \omega}, E^+_k(g_t \omega)) \leq C e^{-\lambda t} \,.
$$
Since the logarithm of the Hodge norm is Lipschitz, we observe that
\begin{align*}
\sigma_k(g_{T}, \bfv_\omega)&= \sigma_k(g_{T}, \bfv_\omega)-\sigma_k(g_{T}, E^+_k(\omega))+ \sigma_k(g_{T}, E^+_k(\omega))\\
&= \sigma_k(g_{T}, E^+_k(\omega))+O(e^{-\lambda t})
\end{align*}
Under the hypothesis of Theorem \ref{dclt}, it follows therefore that 
\begin{align*}
\frac{1}{\sqrt{T}} \left(\sigma_k(g_{T}, E^+_k(\omega))- T(\sum_{i=1}^k \lambda_i) \right)
\end{align*}
satisfies the conclusion of Theorem \ref{dclt} with respect to the measure $\nu$ (instead of $\hat{\nu}$). Now assume that $\bfv\defeq(\omega,\bfv_{\omega})$ is future-Oseledets-generic for $\nu$ a.e. $\omega\in X$ (in the sense of Def. \ref{fut-gen}), then it follows that $\text{proj}_{E^+_k( \omega)}\bfv_{ \omega} \neq 0$ for $\nu$ a.e $\omega\in X$. It follows that 
\begin{align*}
\sigma_k(g_{T}, \bfv_\omega)&= \sigma_k(g_{T}, E^+_k(\omega))+\log(1+O_{\bfv_\omega}(e^{(\lambda'-\lambda)T}))+O_{\bfv_\omega}(1)
\end{align*}
for $0<\lambda'<\lambda$, and this gives us the conclusion of this corollary. 
\end{proof}

\begin{rem} See also \cite[Theorems 3.10, 3.11, 3.21 and 3.22]{AF24} for an abstract central limit theorem for generic sections that implies our corollary. See also \cite[Theorem 4.25]{AF24} for a stronger result that proves a mixing CLT for Oseledets-generic sections of $\bP(\bfH^{(k)})$ (based on our Theorem \ref{dclt} as a crucial input). 
\end{rem}

\begin{rem}By applying the results in \cite{CE13,CF20, AAE+21}, Arana-Herrera and the second-named author derive in  \cite[Theorem 4.32]{AF24} that (measurable) $\SO(2,\bR)$-invariant sections of $\bP(\bfH)$  are future-Oseledets-generic, and we refer to Section 4 of their paper for a thorough discussion on Oseledets genericity of sections and other related results.
\end{rem}
\appendix
\section{Solving Poisson's equation}\label{poisson}

For every $k\in \{1, \dots, h\}$, the group $\SL(2,\R)$ acts on the projective bundle $\bP(\text{\bf{H}}^{(k)})$
by parallel transport on the fibers.  It is then possible
to define a stochastic process on $ \SO(2,\R)\backslash \bP(\text{\bf{H}}^{(k)})$ as a lift by parallel transport of the foliated Brownian motion on $X_*:=\SO(2, \R)\backslash X$.

It is possible to find the solution to the Poisson
equation for the generator~$\laplacian$ based on the semi-group of the stochastic process.  Indeed, if $\{P^{(k)}_t\}$ denotes the semi-group of the stochastic process on the projectivized bundle $\bP(\text{\bf{H}}^{(k)})$ over $X$, given by the formula below. 

For every $x \in X$, the leaf $\SO(2,\R)\backslash \SL(2,\R) x$ can be parametrized by the Poincar\'e
disk $\bD$ with $x$ as the image of the origin. In the following formulas we write integrals on $\SO(2,\R)\backslash \SL(2,\R) x$
with respect to the coordinate $z \in \bD$. 
Let $dA(z)$ denote the (Lebesgue) area on $\bD$. 

\smallskip
For every $\SO(2, \R)$-invariant $f \in L^\infty(\bP(\text{\bf{H}}^{(k)}))$, define
$$
P^{(k)}_t(f) (x,\bfv) := \int_{\bD} p_t(x,z) f(z,\bfv)  dA(z),
$$
where $p_t(x,z)$ denotes the hyperbolic (heat) kernel 
on the leaf $\SO(2, \R) \backslash \SL(2,\R)x$ written on the Poincar\'e disk coordinate $z \in \bD$. 
\begin{rem} On the right-hand side of this definition, $(z, \bfv)$ denotes the value at $z$ of the unique parallel section (with respect to the Gauss--Manin connection) passing through $(x, \bfv)$ on the universal cover $\widetilde{L}_x$ of the leaf $L_x$. Since the universal cover is simply connected (and thus every two points are connected by a unique homotopy class of paths), and the connection is flat, parallel transport is path-independent (i.e., has trivial monodromy), and the definition is independent of the choice of a lift of \((x, \mathbf{v})\). Moreover, since the heat kernel \(p_t(x, z)\) is invariant under isometries of the Poincar\'e disk, and $f$ is defined on the quotient surface $L_x$ (under the action of the (discrete) Veech group), the function $P_t^{(k)}(f)$ descends to a well-defined function on the leaf $L_x$.
\end{rem} 

Then, formally, the Green operator of $\laplacian$ is given by the formula
$$
(G f)(x,\bfv) := \int_0^{+\infty}  P^{(k)}_t (f) dt =  \int_0^{+\infty} \int_{\bD} p_t(x,z) f(z,\bfv)  dA(z) dt\,.
$$
In fact, by a formal calculation
\begin{equation}
\label{eq:Green}
\begin{aligned}
\laplacian (G f)&(x,\bfv) = \int_{\bD} \int_0^{+\infty}  \laplacian_{L_x} p_t(x,z) f(z,\bfv) dt dA(z) \\ & = - \int_{\bD} \int_0^{+\infty} \frac{\partial}{\partial t} p_t(x,z) f(z,\bfv) dy\\  & = 
\lim_{t\to 0} \int_{\bD} p_t(x,z) f(z,\bfv) dA(z) -
\lim_{t\to +\infty} \int_{\bD} p_t(x,z) f(z,\bfv) dA(z)
\\ &= f(x,\bfv) - \lim_{t\to +\infty} \int_{\bD} p_t(x,z) f(z,\bfv) dA(z)\,.
\end{aligned}
\end{equation}
The formula gives a Green operator under the conditions
that 
$$
\begin{aligned}
&\int_0^{+\infty} \big\Vert \int_{\bD} p_t(x,z)  f(z,\bfv)  dA(z) \big\Vert_{L^2 (  \bP(\text{\bf{H}}^{(k)}) ,d\hat{\nu})} dt  < +\infty\,; \\
&\quad \lim_{t\to +\infty} \Vert \int_{\bD} p_t(x,z) f(z,\bfv)  dA(z) \Vert_{L^2(\bP(\text{\bf{H}}^{(k)}), d\hat{\nu})}=0\,.
\end{aligned}
$$
Let us now derive the following lemma: \begin{lemma}\label{brownianlem}
Suppose the (foliated) Laplacian \(\Delta\) has a spectral gap, i.e. for any zero-average $\SO(2, \R)$-invariant function \(f \in W^{2,2}(X,\nu)\),
\[
\Re \int_{X} \overline{f} \Delta f \, d\nu \leq -\frac{1}{C} \int_{X} \vert f \vert^2 \, d\nu
\]
for some constant \(C > 0\). Let  \(P_t\) be the solution semi-group, with domain on $W^{2,2}(X, \nu)$,  of the heat equation 
$$\begin{cases} 
\frac{\partial}{\partial t} P_t f  = \Delta P_t f \\ P_0  = \text{\rm Id} 
\end{cases}$$
Then, for any $\SO(2, \R)$-invariant, zero-average  essentially bounded function \(f \in W^{2,2}(X, \nu)\), we have 
\[
\|P_t f\|_{L^2(X, \nu)} \leq \|f\|_{L^2(X, \nu)} e^{-\frac{t}{C}}\,.
\]
hence the Brownian semigroup \(P_t\) extends to a bounded operator on the subspace $L^2_0(X, \nu)$ of $\SO(2, \R)$-invariant functions  such that the operator norm in the space $\mathcal L(L^2_0(X, \nu))$ of bounded operators on $L^2_0(X, \nu)$ satisfies, for all $t\in \R$, the estimate
$$
\Vert P_t \Vert_{\mathcal L(L^2_0(X, \nu))} \leq e^{-\frac{t}{C}}\,.
$$

\end{lemma}

\begin{proof}
Let $f\in W^{2,2}(X, \nu)$ be any essentially bounded $\SO(2, \R)$-invariant function of zero average and consider the \(L^2\)-norm:
\[
\|P_t f\|_{L^2(X, \nu)}^2 = \int_{X} \vert P_t f\vert^2 \, d\nu \,.
\] Differentiating with respect to 
\(t\in \R\), since  \(\frac{\partial}{\partial t} P_t  = \Delta P_t \) we have
\[
\frac{d}{dt} \|P_t f\|_{L^2(X,\nu)}^2 = 2 \Re \int_{X} \overline{P_t f} \,\frac{\partial}{\partial t} (P_t f) \, d\nu = 2 \Re \int_{X} \overline{P_t f} \,  \Delta (P_t f) \, d\nu.
\]
By the spectral gap assumption, for all zero-average essentially bounded $\SO(2, \R)$-invariant functions $f\in W^{2,2}(X, \nu)$ we have
\[
\Re \int_{X} \overline{P_t f} \, \Delta (P_t f) \, d\nu \leq -\frac{1}{C} \int_{X} \vert P_t f\vert^2 \, d\nu = -\frac{1}{C} \|P_t f\|_{L^2(X, \nu)}^2.
\]
Thus,
\[
\frac{d}{dt} \|P_t f\|_{L^2(X, \nu)}^2 \leq -\frac{2}{C} \|P_t f\|_{L^2(X, \nu)}^2.
\]
Let \(g(t) = \|P_t f\|_{L^2(X,\nu)}^2\). Then, for all $t\in \R$, we have
\[
\frac{dg}{dt} \leq -\frac{2}{C} g(t).
\]
By Gr\"onwall's inequality, the above inequality implies
\[
g(t) \leq g(0) e^{-\frac{2t}{C}}.
\]
that is, by the definition of the function $g$,
\[
\|P_t f\|_{L^2(X, \nu)} \leq \|f\|_{L^2(X, \nu)} e^{-\frac{t}{C}}. 
\] 
The above estimate can then be extended by density of essentially bounded functions in $W^{2,2}(X, \nu)$ to all 
functions (of zero average) in $L^2_0(X, \nu)$.
\end{proof}

\begin{lemma}
\label{lemma:exp}
Let ${P^{(k)}_t}$ denote the semigroup of the lift of the Brownian motion to $\bP(\text{\bf{H}}^{(k)})$. Assume
that the Kontsevich-Zorich Lyapunov exponents satisfy the strict inequality $\lambda_k >\lambda_{k+1}$. Then there exist constants $C>0$ and $\lambda>0$ such that, 
for every $\SO(2, \R)$-invariant function $f\in L^\infty(\bP(\text{\bf{H}}^{(k)}))$, Lipschitz with respect to the Hodge metric on $\bP(\text{\bf{H}}^{(k)})$,  and for all $t>0$, 
$$
\Big\Vert P^{(k)}_t (f) - \int_{\bP(\text{\bf{H}}^{(k)})} f d\hat{\nu} \Big\Vert_{L^2(\bP(\text{\bf{H}}^{(k)}) )} \leq C \Vert f \Vert_{Lip} e^{-\lambda t} \,. 
$$
\end{lemma}
\begin{proof}
 Let $f$ be an $\SO(2,\R)$-invariant bounded Lipschitz  function on $\bP(\text{\bf{H}}^{(k)})$  and let, by integrating with respect to the Lebesgue area $dA(z)$ on $\bD$,
$$
P^{(k)}_t f (x, \bfv)  =  \int_\bD   p_t(x,\bar z) f(\bar z,\bfv) dA(z)  
$$
This function is also $\SO(2,\R)$-invariant since $p_t(x,\bar z)$  is radial, $dA$ is rotationally invariant and
$(\bar z,\bfv)$ is obtained by parallel transport  of $\bfv$, with respect to the flat (Gauss-Manin) connection, from the point  $\SO(2,\R)x$ to the point in the leaf $\SO(2,\R)\backslash SL(2,\R)x$ of coordinate  $z\in \bD$.
 
 For $\nu$-almost all $x \in X$ let
$\rho_x$ denote the hyperbolic Brownian motion starting
at $x$ on the leaf $\SO(2, \R) \backslash \SL(2, \R)x$.
Since $\lambda_{k} > \lambda_{k+1}$, the unstable Oseledets space  $E^+_k$ is well-defined $\nu$-almost everywhere on $X$. For all $t >0$, let then $E^+_k(z)$ denote the
unstable space evaluated at the radial outward unit vector in the circle orbit of coordinate $z\in \bD$. 

By the Oseledets theorem,  for all $v \in \bP(\text{\bf{H}}^{(k)})$, there exist constants $C>0$ and $\lambda >0$ such that, for all $t>0$,
$$
\text{dist}_{\rho_x(t)}  (\bfv, E^+_k(\rho_{x}(t))) \leq C e^{-\lambda t} \,.
$$
Since the function $f$ is Lipschitz, it follows that the function $P^{(k)}_t f(x,\bfv)$ approaches almost everywhere  exponentially fast for diverging $t>0$ the function, defined for $\nu$-almost all $x\in X$,
$$
F_t(x) =  \int_\bD   p_t(x,z) f(z, E^+_k(z) )dA(z) \,, $$
which is  $SO(2,\R)$ invariant and $L^\infty$. By the spectral gap property of the $\SL(2, \R)$ action, or equivalently from the spectral gap property of the foliated Laplacian~$\Delta$, which follows from \cite{AGY06} and \cite{AG13}, Lemma~\ref{brownianlem} implies that there exist constants $C'>0$ and  $\lambda' >0$ such that, for any $\SO(2,\R)$-invariant $L^2$ function $F$  on $X$ (in particular for  $F=F_t$),  for all $s >0$, 
$$
\Vert P^{(k)}_s (F)  - \int_X F d\nu \Vert_{L^2(X, d\nu)}  \leq C' e^{-\lambda' s}  \Vert F \Vert_{L^2(X, d\nu)} \,.
$$
Therefore we have  that
$$
\begin{aligned}
P^{(k)}_{2t} (f)  - \int_{\bP(\text{\bf{H}}^{(k)})} f d\hat{\nu}   &=  P^{(k)}_{2t} (f) -  P^{(k)}_t (F_t) + P^{(k)}_t (F_t)  - \int_{\bP(\text{\bf{H}}^{(k)})} f d\hat{\nu} \\ &=
P^{(k)}_{2t} (f) -  P^{(k)}_t (F_t) + P^{(k)}_t (F_t)  - \int_X F_t d\nu 
\end{aligned}
$$
converges to zero exponentially in $L^2(\bP(\text{\bf{H}}^{(k)}), d\hat{\nu})$ 
since, for all $t>0$,
$$
\begin{aligned}
\Vert P^{(k)}_{2t} (f) &- P^{(k)}_t (F_t) \Vert_{L^2(\bP(\text{\bf{H}}^{(k)}),d\hat{\nu})} = \Vert P^{(k)}_t \big( P^{(k)}_t(f)   -  F_t\big) \Vert_{L^2(\bP(\text{\bf{H}}^{(k)}),d\hat{\nu})} \\ &\leq \Vert  P^{(k)}_t(f)   -  F_t \Vert_{L^\infty(\bP(\text{\bf{H}}^{(k)}),d\hat{\nu})} \leq 
C \Vert f \Vert_{Lip} e^{-t} \,,
\end{aligned}
$$
and 
$$
\begin{aligned}
\Vert P^{(k)}_t (F_t)  &- \int_X F_t d\nu  \Vert_{L^2(\bP(\text{\bf{H}}^{(k)}),d\hat{\nu})} =
\Vert P^{(k)}_t (F_t)  - \int_X F_t d\nu  \Vert_{L^2(X,d\nu)} \\
& \leq C' \Vert F_t \Vert_{L^2(X, d\nu)} e^{-\lambda' t} \leq C' \Vert f \Vert_{L^\infty(\bP(\text{\bf{H}}^{(k)}), d\hat{\nu})} e^{-\lambda' t}\,.
\end{aligned}
$$
We remark that by the Fubini theorem and by the $SL(2,\R)$-invariance of the measure $\nu$, for all $t>0$ we have 
$$
\begin{aligned}
\int_X F_t d\nu&=   \int_X \int_\bD   p_t(x,z) f(z, E^+(z) ) dA(z)   d\nu  \\ &= \int_X  \int_0^{+\infty} \frac{1}{2\pi} \int_0^{2\pi}   p_t(r) f(g_r r_\theta x, E^+(r_\theta x) ) d\theta dr   d\nu \\& =  \frac{1}{2\pi} \int_0^{+\infty}   p_t(r)  \int_X  f((x, E^+(r_\theta(x)) )  d\nu d\theta dr  = \int_{\bP(\text{\bf{H}}^{(k)})}  f d\hat{\nu}
\end{aligned}
$$
The argument is complete as the stated exponential $L^2$
convergence follows.
\end{proof}

\begin{quest}\label{expmix} Is the projective Kontsevich-Zorich cocycle on $\bP(\text{\bf{H}}^{(k)})$ {\it above the Teichm\"uller geodesic} flow exponentially mixing
(under the hypothesis of Lemma~\ref{lemma:exp} on the Kontsevich--Zorich spectrum)? Our result establishes exponential mixing for the projective Kontsevich-Zorich cocycle above Brownian motion, but our argument breaks down for the deterministic cocycle (above the Teichm\"uller geodesic flow).

\end{quest}

Let $\laplacian$ the generator of the foliated Brownian
motion on $\SO(2, \R) \backslash X$.  The operator $\laplacian$ extends naturally to an operator on $\SO(2, \R) \backslash \bP(\text{\bf{H}}^{(k)})$ which is the generator of the lift of the Brownian motion by parallel transport. Let $W^{2,2} \subset L^2(\bP(\text{\bf{H}}^{(k)}), d\hat{\nu} )$ denote the domain
of $\laplacian$ which consists of all square-integrable functions which are differentiable up to second order along the leaves of the foliation of $\SO(2, \R) \backslash \bP(\text{\bf{H}}^{(k)})$ by hyperbolic disks.

 We can now derive the following corollary (by an argument, given below, analogous to the proof of \cite[Corollary 1]{L95}):

\begin{lemma}\label{lemma:Poisson} For any $\SO(2, \R)$-invariant zero-average bounded Lipschitz function $f:(\bP(\text{\bf{H}}^{(k)}), d\hat{\nu}) \to \C$, the Poisson equation
$$
\laplacian u = f 
$$
has a unique solution $u \in W^{2,2}(\bP(\bfH^{(k)}), \hat \nu)$.
\end{lemma}
\begin{proof} 
Let $u \in L^2(\bP(\text{\bf{H}}^{(k)}), d\hat{\nu})$ be defined as
$$
u(x,\bfv) = \int_0^{+\infty} P^{(k)}_t(f) (x,\bfv) dt  := \int_0^{+\infty} \int_{\bD} p_t(x,z) f(z,\bfv)  dA(z) dt \,.
$$
The improper integral converges in $L^2$ by Lemma~\ref{lemma:exp} and by the assumption that $f$
has zero average. It is a solution of the Poisson equation by the calculation in formula \eqref{eq:Green}
and since, again by Lemma~\ref{lemma:exp},
$$
P^{(k)}_t(f) \to 0   \quad \text{ in }  L^2(\bP(\text{\bf{H}}^{(k)}), d\hat{\nu})\,.
$$
The solution of the Poisson equation is unique up to additive constants since, by ergodicity of the Teichm\"uller geodesic flow and by the Oseledets theorem, the foliation of $\SO(2, \R) \backslash \bP(\text{\bf{H}}^{(k)})$, with leaves given by the projections of $\SL(2, \R)$ orbits, is ergodic (with respect to the measure $\hat \nu$).  In fact, by \cite{Gar83}, Theorem 1 (b), any bounded Borel function which is harmonic on each leaf must be constant on almost all leaves, relative to any finite harmonic measure. (Note that the theorem is stated for compact manifolds, but proved, as remarked in \cite{Gar83}, for
any manifold provided that the foliation satisfies a condition of bounded geometry for the leaves, stated in \cite{Gar83}, \S 2, which holds in our case).

Finally, since the operator $\laplacian$ is elliptic along the leaves of the foliation, which are of the form $\SO(2, \R) \backslash \SL(2, \R) (x,\bfv)$ with $(x,\bfv)\in \bP(\text{\bf{H}}^{(k)})$, and since the function $f$ is bounded, hence $f \in L^2 (\bP(\bfH^{(k)}), \hat \nu)$, and smooth along the leaves of the foliation, it follows that the solution $u$ belongs to the space $W^{2,2}(\bP(\bfH^{(k)}), \hat \nu)$ and it is smooth along the leaves of the foliation.
\end{proof}

\section{A Martingale Central Limit Theorem}

We prove below for the convenience of the reader a 
martingale central limit theorem (CLT)  adapted to our setting.
A version of this result for the Brownian motion on hyperbolic manifolds appears in the work of Franchi and Le Jan (see \cite[Lemma VIII.7.4]{FL12}). 

\begin{thm}\label{flj:mclt} Let \((M_t)\) be a real-valued continuous martingale of the form  

\[
M_t =
\int_0^t \sigma_1(\rho_s) \, dW_s^{(1)} + \int_0^t \sigma_2(\rho_s) \, dW_s^{(2)}.
\]
for some $\SO(2, \R)$-invariant functions \(\sigma_1, \sigma_2\ \in L^2(\bP(\text{\bf{H}}^{(k)}), d\hat{\nu})\). Then the law of \(M_t / \sqrt{t}\) converges, as \(t \to \infty\), towards the centered Gaussian law with variance  

\[
V\defeq \int_{\bP(\text{\bf{H}}^{(k)})} H^2 \, d\hat\nu, \quad \text{where } H^2 := \sigma_1^2+\sigma_2^2.
\]

In addition, if \( (\tau_t)_{t \geq 0} \) is an increasing, adapted time change with \( \tau_t/t \to 1 \) almost everywhere,
then the law of \(M_{\tau_t} / \sqrt{t}\) converges in distribution, as \(t \to \infty\), to the same centered Gaussian law.

\end{thm}

\begin{proof}
Define the two-dimensional process  \(\mathbf{X}_t\defeq(X_t^{(1)}, X_t^{(2)})\) such that
\[ 
X^{(1)}_t = \int_0^t \sigma_1(\rho_s) \, dW_s^{(1)}, \quad X^{(2)}_t = \int_0^t \sigma_2(\rho_s) \, dW_s^{(2)}. 
\] 
Since \( W^{(1)} \) and \( W^{(2)} \) are independent standard Brownian motions, \( X^{(1)}_t \) and \( X^{(2)}_t \) are uncorrelated local martingales. Their joint quadratic variation matrix is 
\[ 
\langle \mathbf{X} \rangle_t = \begin{bmatrix} \langle X^{(1)} \rangle_t & \langle X^{(1)}, X^{(2)} \rangle_t \\ \langle X^{(1)}, X^{(2)} \rangle_t & \langle X^{(2)} \rangle_t \end{bmatrix} = \int_0^t \begin{bmatrix} \sigma_1^2(\rho_s) & 0 \\ 0 & \sigma_2^2(\rho_s) \end{bmatrix} ds. 
\]
By Knight’s theorem \cite[Theorem V.1.9]{RY}, there exist time changes 
\[ 
T_t^{(1)} = \inf \{ s : \langle X^{(1)} \rangle_s > t \}, \quad T_t^{(2)} = \inf \{ s : \langle X^{(2)} \rangle_s > t \} 
\] such that the time-changed processes \[ B_t^{(1)} \overset{d}{=} X^{(1)}_{T_t^{(1)}}, \quad B_t^{(2)} \overset{d}{=} X^{(2)}_{T_t^{(2)}} \] are the components of a standard 2-dimensional Brownian motion, hence \(M_t\) is distributed as a sum of two time-changed independent Brownian motions \(B_{\langle X^{(1)} \rangle_t}^{(1)}\) and \(B_{\langle X^{(2)} \rangle_t}^{(2)}\).  
By the ergodic theorem, for $i=1,2$, 
\[
V_i :=  \lim_{t\to +\infty} \frac{\langle X^{(i)} \rangle_t}{t}
= \lim_{t\to +\infty} \frac{1}{t}\int_0^t \sigma_i^2(\rho_s) ds = \int_{\bP(\text{\bf{H}}^{(k)})}  \sigma_i d \hat \nu \,,
\]
hence the limit
\[
\lim_{t\to +\infty} \frac{M_t}{\sqrt{t}}  \overset{d}{=} \lim_{t\to +\infty}
\Big( B_{\frac{\langle X^{(1)} \rangle_t}{t}}^{(1)} + B_{\frac{\langle X^{(2)} \rangle_t}{t}}^{(2)}    \Big) 
\overset{d}{=} B^{(1)}_{V_1} + B^{(2)}_{V_2} \,,
\]
has a normal distribution with zero mean and variance $V=V_1 + V_2$ as the sum of the two normally distributed, independent
random variables $B^{(1)}_{V_1}$ and $B^{(2)}_{V_2}$ with zero mean and variance $V_1$ and $V_2$, respectively. Thus $M_t/\sqrt{t}$ converges in distribution to the normal distribution $\mathcal{N}(0,V)$, completing the proof of the first statement.

As for the second statement, since, for $i=1,2$; we have
\[
\lim_{t \to \infty} \frac{\langle X^{(i)} \rangle_{\tau_t}}{t} = \lim_{t \to \infty}  \frac{\tau_t}{t} \frac{\langle X^{(i)} \rangle_{\tau_t}}{\tau_t} = \int_{\mathbb{P}(\mathbf{H}^{(k)})} \sigma_i^2 \, d\hat{\nu} = V_i \,,
\]
we have, as above
\[
\lim_{t\to +\infty} \frac{M_{\tau_t}}{\sqrt{t}}  \overset{d}{=} B^{(1)}_{V_1} + B^{(2)}_{V_2} \overset{d}{=} \mathcal{N}(0, V) \,,\] 
as claimed. The proof is therefore complete.
\end{proof}

\bibliography{mybib}
\bibliographystyle{amsalpha}
\end{document}